\documentclass[10pt,a4paper,oneside]{article}

\usepackage{amssymb}
\usepackage{amsmath,amsfonts,amsthm,amstext}
\usepackage[english]{babel}
\usepackage{inputenc}
\usepackage{csquotes}
\usepackage{makeidx}
\usepackage{fancybox}
\usepackage{niceframe}
\usepackage{enumerate}
\usepackage{array}
\usepackage{subfigure}
\usepackage{newlfont}
\usepackage{verbatim}
\usepackage[dvips]{psfrag}
\usepackage{graphicx,color,soul}
\usepackage[usenames,dvipsnames,svgnames,table]{xcolor}
\usepackage{xcolor}
\usepackage{float}
\usepackage[footnotesize]{caption}
\usepackage{indentfirst}
\usepackage[normalem]{ulem}
\usepackage{pst-text}
\usepackage{pgf}
\usepackage{wrapfig}
\usepackage{lmodern}
\usepackage[top=2cm,bottom=1.5cm,right=1.5cm,left=2cm]{geometry}
\usepackage{xfrac}
\usepackage{aliascnt}
\usepackage{hyperref}
\usepackage{framed}
\usepackage[ampersand]{easylist}
\usepackage{pdfpages}
\usepackage{afterpage}

\usepackage[backend=bibtex,style=alphabetic]{biblatex}
\allowdisplaybreaks

\input xy
\xyoption{all}

\graphicspath{{./FIG/}{./FIG1/}}

\newcommand{\s}{\Sigma}
\newcommand{\Z}{Z=(X,Y)}
\newcommand{\U}{\mathcal{U}}
\newcommand{\R}{\mathbb{R}}

\newcommand{\sgn}[1]{\mbox{sgn}\left( #1 \right)}
\newcommand{\0}{\mathbf{0}}

\newtheorem{definition}{Definition}[section]

\newaliascnt{defi}{definition}
\newtheorem{defi}[defi]{Definition}
\aliascntresetthe{defi}

\newaliascnt{theo}{definition}
\newtheorem{theo}[theo]{Theorem}
\aliascntresetthe{theo}

\newaliascnt{corol}{definition}
\newtheorem{corol}[corol]{Corolary}
\aliascntresetthe{corol}

\newaliascnt{lemma}{definition}

\aliascntresetthe{lemma}

\newaliascnt{prop}{definition}
\newtheorem{prop}[prop]{Proposition}
\aliascntresetthe{prop}

\newaliascnt{rem}{definition}
\newtheorem{rem}[rem]{Remark}
\aliascntresetthe{rem}

\theoremstyle{definition}

\numberwithin{equation}{section}

\let\oldtheequation\theequation
\makeatletter
\def\tagform@#1{\maketag@@@{\ignorespaces#1\unskip\@@italiccorr}}
\renewcommand{\theequation}{(\oldtheequation)}
\makeatother

\title{Generic behavior of a piecewise smooth vector field with non-smooth switching surface}
\author{Juliana Larrosa\footnote{ \footnotesize Departamento de Matem\'{a}tica, Universidade Federal de Santa Maria, Avenida Roraima 1000, 97195-000  Santa Maria, RS, Brasil.} , Marco Antonio Teixeira\footnote{\footnotesize Departamento de Matem\'{a}tica, Universidade Estadual de Campinas, Rua S\'{e}rgio Buarque de Holanda 651, Cidade Universit\'{a}ria, 13083-859, Campinas, SP, Brazil.}, Tere M-Seara\footnote{\footnotesize Dit. Matem\`{a}tiques, Universitat Polit\`{e}cnica de Catalunya, Diagonal 647, 08028 Barcelona,Spain.}} 

\begin{document}

\maketitle

\begin{abstract}
This paper consists in discussing some issues on generic local classification  of  typical singularities of $2D$ piecewise smooth vector fields  when the switching set is an algebraic variety. The main focus is to obtain classification results  concerning structural stability and generic codimension one bifurcations. 
\end{abstract}

\section{Introduction}

First of all, we observe that this paper is part of a general program involving the study of  discontinuous piecewise smooth systems in $\R^n$ of the form
\begin{equation} \label{eq:intro}
	\dot{\mathbf{x}} = F(\mathbf{x}) + \sgn{f(\mathbf{x})}G(\mathbf{x});
\end{equation}
where $\mathbf{x} = (x_1,....., x_n)$ and $F, \,G : \R^n \rightarrow \R^n$, $f: \R^n \rightarrow \R$  are smooth functions. Note that we have two different smooth systems, one, $X=F+G$, in the half space defined by $f(\mathbf{x}) > 0$ and the other, $Y=F-G$, in the half space defined by $f(\mathbf{x}) < 0$.

In this direction, let $X,Y: \U \subset \R^2 \rightarrow \R^2$ be sufficiently smooth vector fields defined in a bounded neighborhood $\U$ of the origin. Consider $f: (x_1,x_2) \in \mathcal{U} \subset \R^2 \rightarrow x_2 \in \R$ having $0$ as a regular value and let $\s=f^{-1}(0)$. Thus $\s$ is a regular codimension one submanifold of $\U \subset \R^2$. The submanifold $\s$ splits the open set $\U$ into two open sets $\U^+ = \{ p \in \U: x_2>0 \}$ and $\U^- = \{ p \in \U: x_2<0 \}$.

Let $\mathcal{Z}$ be the space of all piecewise smooth vector fields $Z=(X,Y)$ defined as \begin{equation} \label{sys1}
Z(p)=\begin{cases} X(p), & \mbox{if } \, p \in \U^+ \\ Y(p), & \mbox{if } \, p \in \U^- \end{cases}.
\end{equation}

The dynamics on each open set $\U^\pm$ is given by the smooth vector fields $X$ and $Y$, respectively.  The submanifold $\s$ is called {\it discontinuity curve} or {\it switching curve} and we assume that the dynamics over $\s$ is given by the Filippov's convention. More details about the Filippov's convention can be found in \cite{Filippov}. The piecewise smooth vector field defined in this way is called a {\it Filippov system.}

Due to its importance on applications, Filippov systems have been largely studied in the recent years. There are a huge number of works focusing on local and global aspects of these systems. For some works dedicated to planar Filippov systems, one can see \cite{Kuznetsov}, \cite{mst} and \cite{dercole} and for examples on higher dimensions, look at \cite{mike1} and \cite{dercole1}.

Concerning the study of the generic behavior for planar Filippov systems, it was firstly made by Kozlova (\cite{koz}). 
In \cite{Kuznetsov}, Kuznetsov at al., classified and studied all the codimension one bifurcations and also some global bifurcations. Guardia, Seara and Teixeira, in \cite{mst}, complemented these works presenting a rigorous proof of the theorem which classifies the set of the local $\s-$structural stable Filippov systems and revisited the codimension one bifurcations. In addition, they gave a preliminary classification of codimension two bifurcations. 

Using the concepts of local $\s-$equivalence and weak equivalence of unfoldings, in \cite{cjt} the authors revisited the codimension one generic local bifurcations of Filippov systems, presenting a rigorous classification of the generic fold-fold singularities set $\Lambda^F$, as well as a formal study of the versal unfoldings of each singularity. 
Moreover, they have proved that $\Lambda^F$ is a codimension one embedded submanifold of $\mathcal{Z}$. \newline

The most part of the works on piecewise smooth systems are devoted to study switching surfaces which are regular curves or surfaces. However, using the same definition of piecewise smooth vector fields, one can consider $\s$ as the  union of two  codimension one submanifolds which intersect transversely at the origin. 

In this context, we are interested in the particular subject: $2D$ systems as in \ref{eq:intro} for which $f(x_1,x_2)=x_1 \cdot x_2$. We understand that a systematic programme towards the bifurcation theory for such systems are currently emergent.

In this case,  the origin is a non-degenerated critical point of $f$ and $\s=f^{-1}(0)$ is a degenerate hyperbole. Thus $\s$ can be seen as the union of two lines that intersect transversally at the origin. The piecewise smooth vector field $Z=(X,Y)$ is given exactly as in $\ref{sys1}$ and in this case, we assume that the Filippov's convention is valid in  $\s \setminus \{\0 \}$.  

Piecewise dynamical systems naturally arise in the context of many applications.  In this direction our approach is motivated by equations expressed as
$$ \ddot{ \mathbf{x}} + a \dot{\mathbf{x}} = \sgn{\mathbf{x}\cdot f(\mathbf{x},\dot{ \mathbf{x}}}));
$$ 
that are commonly found in many fields such as Control Theory and Engineering. In \cite{Bar} problems involving asymptotic stability  of such systems are fairly discussed. In this book the author presents an stabilization problem that can be solved provided a discontinuity of this type is introduced in the system. In \cite{dieci}, the authors propose a special choice for the sliding vector field at the intersection of the two manifolds.

Let $\Omega$ be the set of all piecewise vector fields defined as above. In this work, our aim is to describe rigorously the set of the locally  $\s-$structurally stable vector fields having this kind of switching set and their codimension one bifurcations. The structure of this work is as follows:

In \autoref{sec1} we recall the objects we are going  to work with, as the trajectories, tangencies,  notion of local $\s-$equivalence and weak equivalences of unfoldings.

Section \ref{sec:stability} is devoted to give a classification of the set $\Omega_0$ composed by the locally $\s-$structurally stable vector fields in $\Omega$. To give this classification we establish the generic conditions which are necessary to $Z \in \Omega$  be locally $\s-$structurally stable. Moreover, in each case we construct the $\s-$equivalence between $Z \in \Omega_0$ and its corresponding ``normal form''. 

Once we have classified the set $\Omega_0 \subset \Omega$, in \autoref{sec:cod1} we describe the set $\Xi_1 \subset \Omega_1 = \Omega \setminus \Omega_0$ composed by the vector fields $Z$ which have a codimension one bifurcation at the origin. In order to get this result, we establish some conditions for $Z \in \Xi_1$ and show that $\Xi_1$ is open in $\Omega_1$ (endowed with the induced topology of $\Omega$). More precisely, $\Xi_1$ is an embedded codimension one submanifold of $\Omega$. Finally, we show that for each $Z \in \Xi_1$ the unfoldings which are transverse to $\Xi_1$ at $Z$ are weak equivalent.

\section{First definitions and results} \label{sec1}

Let $f: \R^2,\0 \rightarrow \R,0$ be a $C^r$ smooth function such that $f(\0)=0$ and that $0$ a non degenerate critical point. The purpose of this work is to study the generic singularities of planar piecewise vector fields $Z$ which discontinuity set is given by the zeros of the map $f(x_1,x_2)$.

As it is known that there are coordinates around the origin such that  $f$ can be written as $f(x_1,x_2)=x_1^2 \pm x_2^2$. In this paper we will study the case $f(x_1,x_2)=x_1^2 - x_2^2$, or equivalently, $f(x_1,x_2)=x_1 \cdot x_2$.

\bigskip

In this direction, let $X$ and $Y$ be smooth $\mathcal{C}^r$, $r \geq 1$ vector fields defined in a bounded neighborhood $\U$ of the origin. 

Let $f: p=(x_1,x_2) \in \U \mapsto f(p)=x_1 x_2 \in \R$. The set $\s = f^{-1}(\0)$ is an algebraic variety and splits $\U$ as the closure of the regions  $\U^+=\{ p \in \R^2 : f(p)>0 \}$ and $\U^-=\{ p \in \R^2 : f(p)<0 \}$. Moreover, the region $\U^+$ can be decomposed as $\U^+_{+} = \U^+ \cap \{ x_2 >0 \} $ and $\U^+_{-} = \U^+ \cap \{ x_2 <0 \}$, as can be seen in \autoref{fig:disc01}. Analogously, we define the regions $\U^-_+$ and $\U^-_-.$

Let $\Omega$ be the set of all piecewise vector fields defined as: \begin{equation} Z(p)=  \begin{cases}
		X(p), & \mbox{if} \quad  p \in \U^+ \\
		Y(p), & \mbox{if} \quad  p \in \U^-
	\end{cases}.
\end{equation}
The set $\s$ is the discontinuity set or switching set. Observe that we can write $ \s = \s _1 \cup \s _2$ with $\s_1 = \{ (x_1,x_2) \in \s : x_1=0 \}$ and $\s_2 =  \{ (x_1,x_2) \in \s : x_2=0 \}$. Moreover, $\s_1 = \overline{\s_1^+ \cup \s_1^-}$ where $\s_1^+ = \{ (0,x_2) \in \s_1 : x_2 >0 \}$ and $\s_1^- = \{ (0,x_2) \in \s_1: x_2 <0 \}$. Similarly, one can write $\s_2 = \overline{\s_2^+ \cup \s_2^-}$.

\begin{figure}[!htbp]
	\centering
	\begin{tiny}
		\def\svgwidth{4cm}
		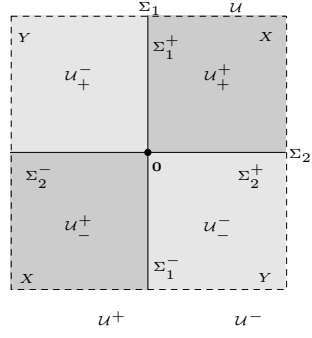
	\end{tiny}
	\caption{The decomposition of the domain $\U$.}
	\label{fig:disc01}
\end{figure}
Each $\s_i$, for $i=1,2$, can be decomposed as the closure of the crossing, sliding and escaping regions as follows: \begin{eqnarray*}
	\s_i^c &=& \left\{ p \in \s_i : X_i \cdot Y_i (p) >0 \right\}, \\
	\s_i^s &=& \left\{ p \in \s_i^+ : X_i(p)<0 , \, Y_i(p)>0 \right\} \cup \left\{ p \in \s_i^- : X_i(p)>0 , \, Y_i(p)<0 \right\}, \\
	\s_i^e &=& \left\{ p \in \s_i^+ : X_i(p)>0 , \, Y_i(p)<0 \right\} \cup \left\{ p \in \s_i^- : X_i(p)<0 , \, Y_i(p)>0 \right\}.
\end{eqnarray*} 
Then the crossing, sliding and escaping regions in $\s$ are the union of the corresponding regions in $\s_1$ and $\s_2$. 

In the regions $\s_i^{s,e}$, for $i=1,2$, we define the sliding vector field:
\begin{equation} 
	Z_i^s(p)=\displaystyle{\frac{1}{Y_i(p)-X_i(p)} [ Y_i(p)X_j(p) -X_i(p)Y_j(p)]} \bigg|_{\s_i}, \label{eq:slidingdef} 
\end{equation} where $i,j=1,2$ and $i\neq j.$

\begin{defi}
	Fix $i=1,2$. The point $p\in \s_i^{s,e}$ is a pseudo-equilibrium if $Z_i^s(p)=0$ and it is hyperbolic if $(Z_i^s)'(p) \neq 0$.
\end{defi}


Since we are interested on low codimension singularities, we focus our attention just to one kind of tangency, the fold point, which is defined below.

\begin{defi} \label{def:regfold}
	The point $p \in \s_i$ is a {\bf fold point} of $X$ in $\s_i$ if $X_i(p)=0$ and $\displaystyle{X_j (p) \cdot \frac{\partial}{\partial x_j}X_i(p) \neq 0}$ for $i,j=1,2$ and $i\neq j$. 
	Moreover, $p \in \s$ is a {\bf regular-fold} of $X$ in $\s_i$ if is a fold point for $X$ in $\s_i$ and $Y$ is transverse to $\s$ at the origin, that is, $Y_l(\0) \neq 0$ for $l=i,2$.
	
	Analogously, we define a fold point and a regular fold for $Y$. 
\end{defi}

\begin{figure}[!htbp]
	\centering
	\begin{tiny}
		\def\svgscale{0.35}
		\subfigure[]{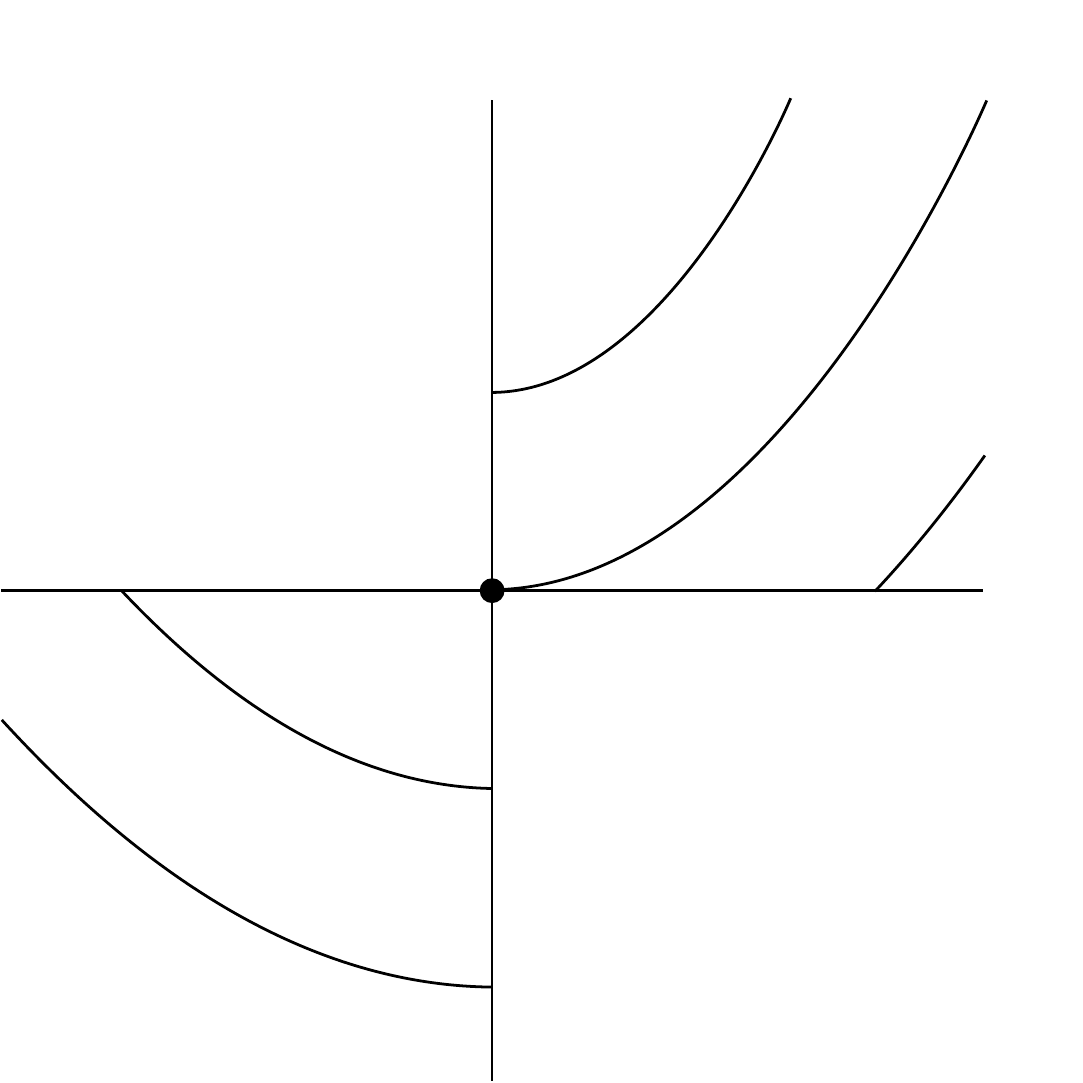 \label{fold02}}
		\hspace{2cm}
		\def\svgscale{0.35}
		\subfigure[]{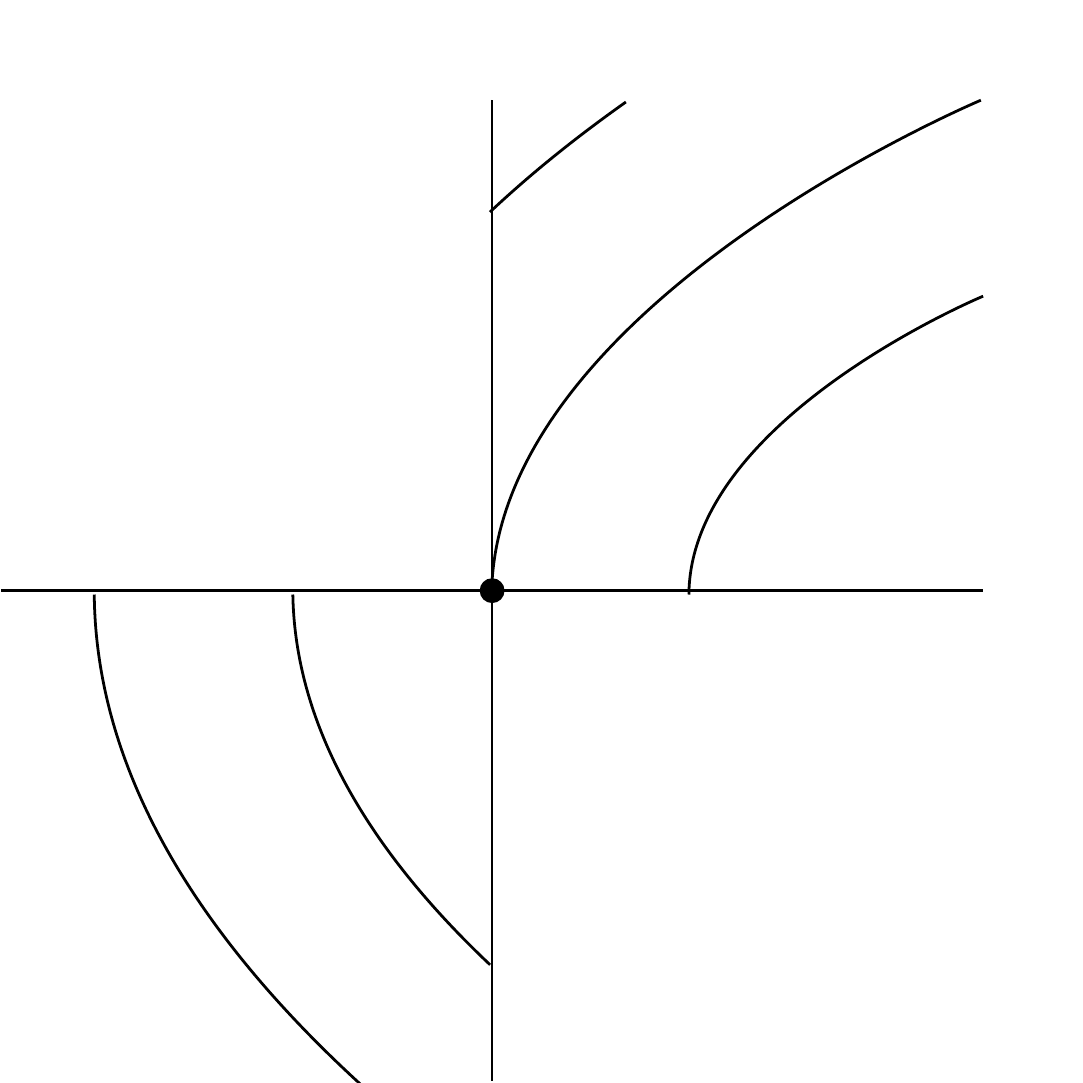 \label{fold01}}
	\end{tiny}
	\caption{The origin is a fold point of $X$. \subref{fold02} $X_2(\0)=0$ and $X_1 \cdot \frac{\partial}{\partial x_1} X_2 (\0) >0 $; \subref{fold01} $X_1(\0)=0$ and $X_2 \cdot \frac{\partial}{\partial x_2} X_1 (\0) >0 $.}
	\label{fig:folds}
\end{figure}


Now we define the trajectories of $Z$ through points of $\U$, following \cite{mst}. Let us denote by $\varphi_X(t;p)$ the flow of a regular vector field $X$. In order to preserve the uniqueness of orbits, we assume that, if $p \in \U^\pm$ is such that the curve $\{\varphi_{X,Y}(t;p); t \in \R \} \pitchfork \s^{s,e}$, then the trajectories of $X$ and $Y$ through $p$ are relatively open, that is, they do not reach $\s^e \cup \s^s$ in finite time.

Next definition gives the trajectories through a point of $\U \setminus \{ 0 \}$.

\begin{defi} Let $p \in \U \setminus \{ 0\}$, then its trajectory $\varphi_Z(t;p)$ is given by:
	\begin{itemize}
		\item If $p \in \U^+ \cup \U^-$ then its trajectory is given by the trajectory of $X$ or $Y$, respectively.
		\item If $p \in \s^c=\s_{1}^c \cup \s_{2}^c$ then its trajectory is the concatenation of its respective trajectories in $\U^+$ and $\U^-$;
		\item If $p \in \displaystyle{(\s^s_i \cup \s^e_i) \setminus \{ 0 \}}$ for $i=1 \mbox{ or } 2$, then its trajectory is given by $\varphi_{Z^s_i}(t;p)$, where $Z_i^s$ is given in \ref{eq:slidingdef};
		\item If $p \in \partial \s^e_i \cup \partial \s^s_i \cup \partial \s^c_i$ and if $\displaystyle{\lim_{q \rightarrow p^-} \varphi_Z(t;q) = \lim_{q \rightarrow p^+} \varphi_Z(t;q)}$, then $\varphi_Z(t;p)=\displaystyle{\lim_{q \rightarrow p} \varphi_Z(t;q)}$. These points are \textbf{regular tangency points}.
		\item If $p \in  \partial \s^e_i \cup \partial \s^s_i \cup \partial \s^c$ does not satisfy the last condition, then $\varphi_Z(t;p)= \{ p \}$ and it is called by \textbf{singular tangency points}.
	\end{itemize}
\end{defi}

Next definition gives the trajectory of $p=\0.$

\begin{defi} Let $\{ \0 \} \in \s$, then its trajectory $\varphi_Z(t;\0)$ is described below: 
	\begin{itemize}
		\item If $\{ \0 \} \in \s=\overline{\s^c}=\overline{\s_1^c \cap \s_2^c}$ then there is only one trajectory of $X$ or $Y$ which cross the origin and we define $\varphi_Z(t;\0)$ as being this trajectory;
		\item If $\{ \0 \} \in \overline{\s^e_i \cup \s^s_i} \cap \overline{\s^c_j }$, $i,j =1,2$ and $i \neq j$, then the trajectory $\varphi_Z(t;\0)=\varphi_{Z^s_i}(t; \0)$;
		\item If $\{ \0 \} \in \cap_{i=1}^{2} \overline{\s^e_i \cup \s^s_i}$ then $\varphi_Z(t;\0)= \{ \0 \}$;
	\end{itemize}
\end{defi}

After the last two definitions, we can define the singularities of $Z \in \Omega$.

\begin{defi} The singularities of $Z \in \Omega$ are:
	\begin{itemize}
		\item $p \in \overline{\U^\pm}$ which are singularities of $X$ or $Y$, respectively;
		\item $p \in \s^{e,s}_i$ such that $Z^s_i(p)=0$;
		\item $p \in \partial \s^e_i \cup \partial \s^s_i \cup \partial \s^c_i$, for $i=1,2$, which are singular tangency points;
		\item $\{ \0 \} \in \s$ when both $Z^s_i$ are defined in a neighborhood of the origin, $i=1,2$. 
	\end{itemize}
\end{defi}

In the sequel we define some different types of ``periodic'' orbits which can appear in piecewise smooth systems, once again we follow the definitions given in \cite{mst}.

\begin{defi}
	A regular periodic orbit is a regular orbit $\gamma = \{ \phi_Z(t; p): t \in R \}$, which belongs to $\U^+ \cup \U^- \cup \s^c$ and satisfies $\phi_Z(t+T;p)=\phi_Z(t;p)$ for some $T >0$.
\end{defi}

\begin{figure}[!htbp]
	\centering
	\begin{tiny}
		\def\svgscale{0.4}
		\subfigure[A regular periodic cycle]{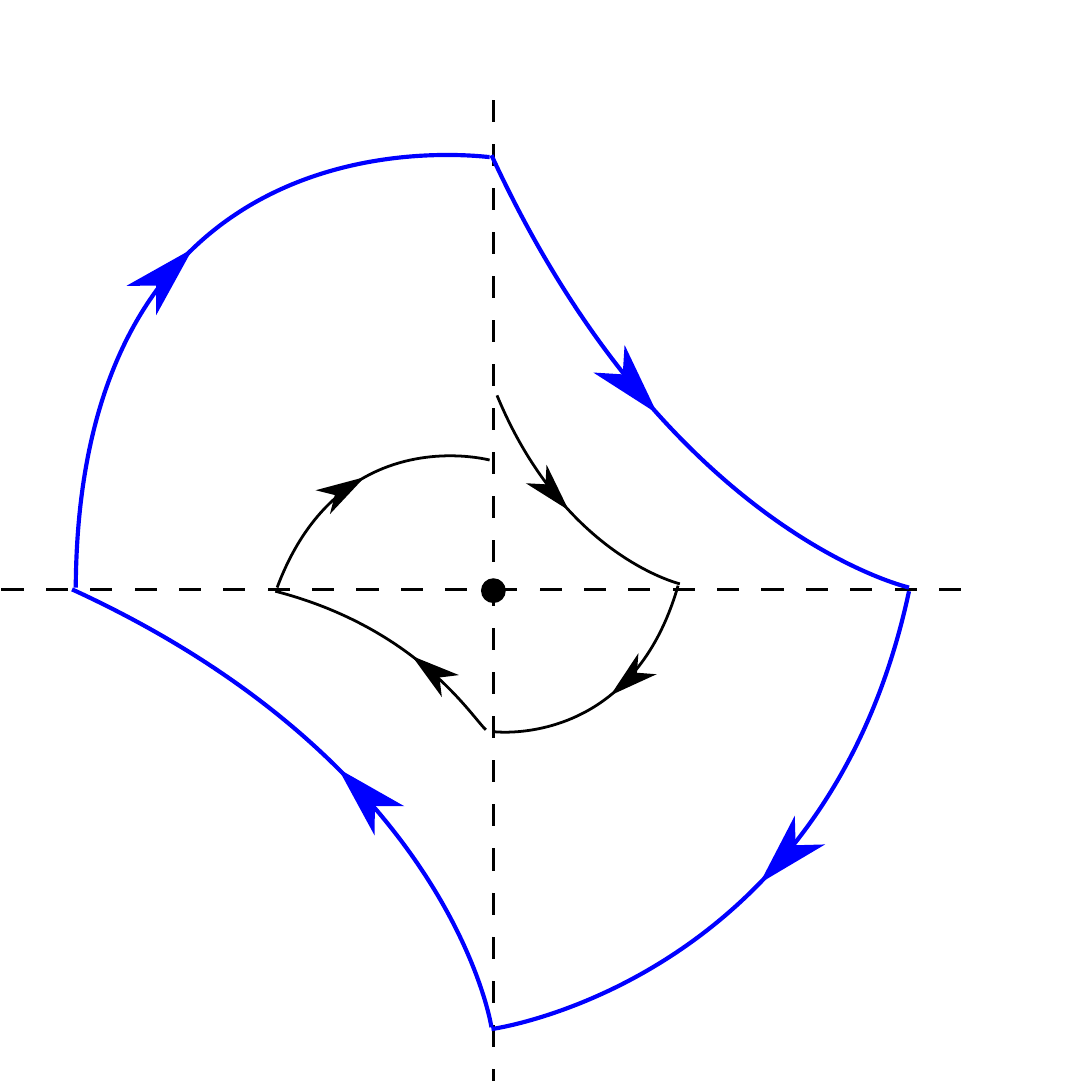 \label{fig:cro1}}
		\qquad \qquad
		\def\svgscale{0.4}
		\subfigure[A periodic cycle]{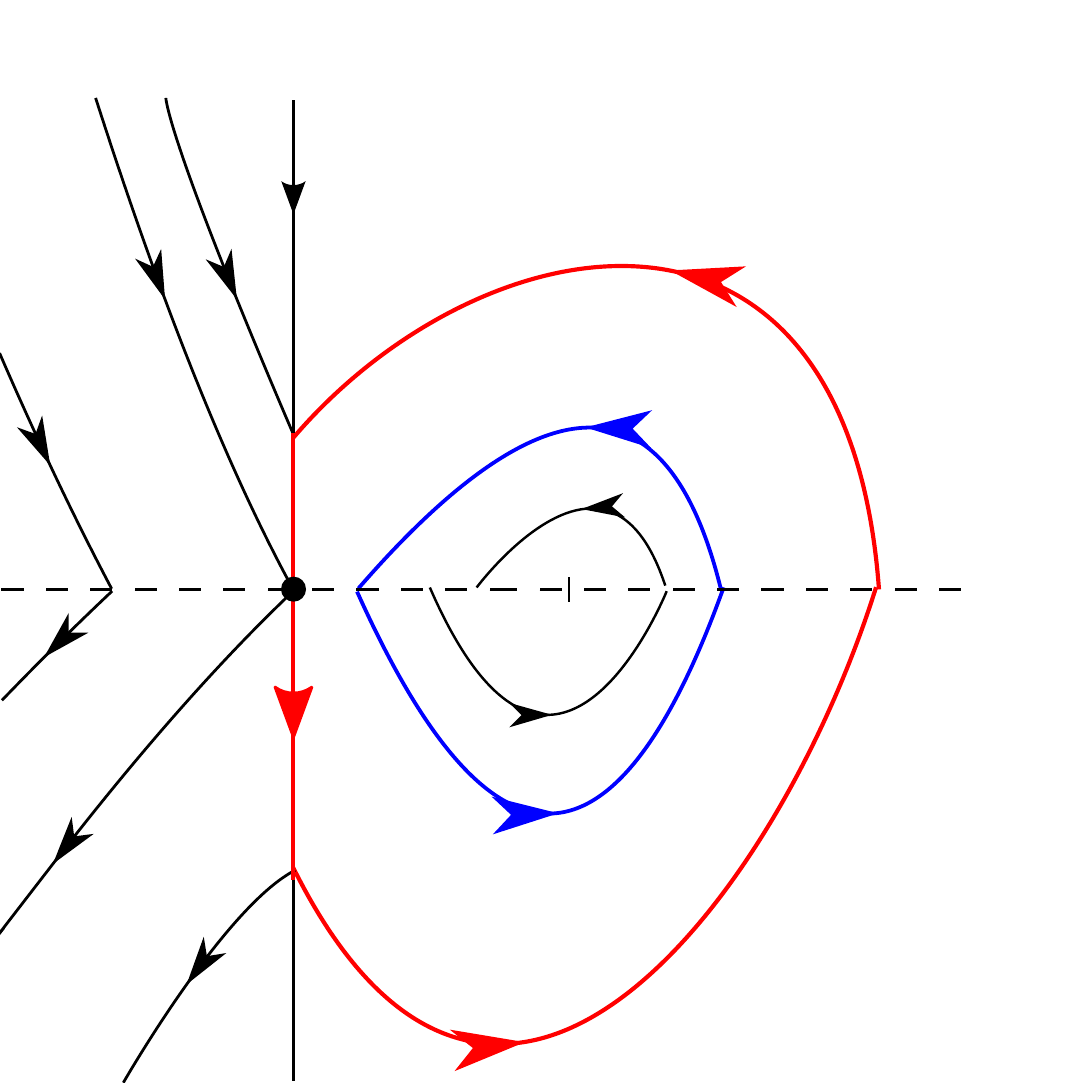 \label{fig:cro}}
		\qquad \qquad
		\def\svgscale{0.4}
		\subfigure[A pseudo-cycle]{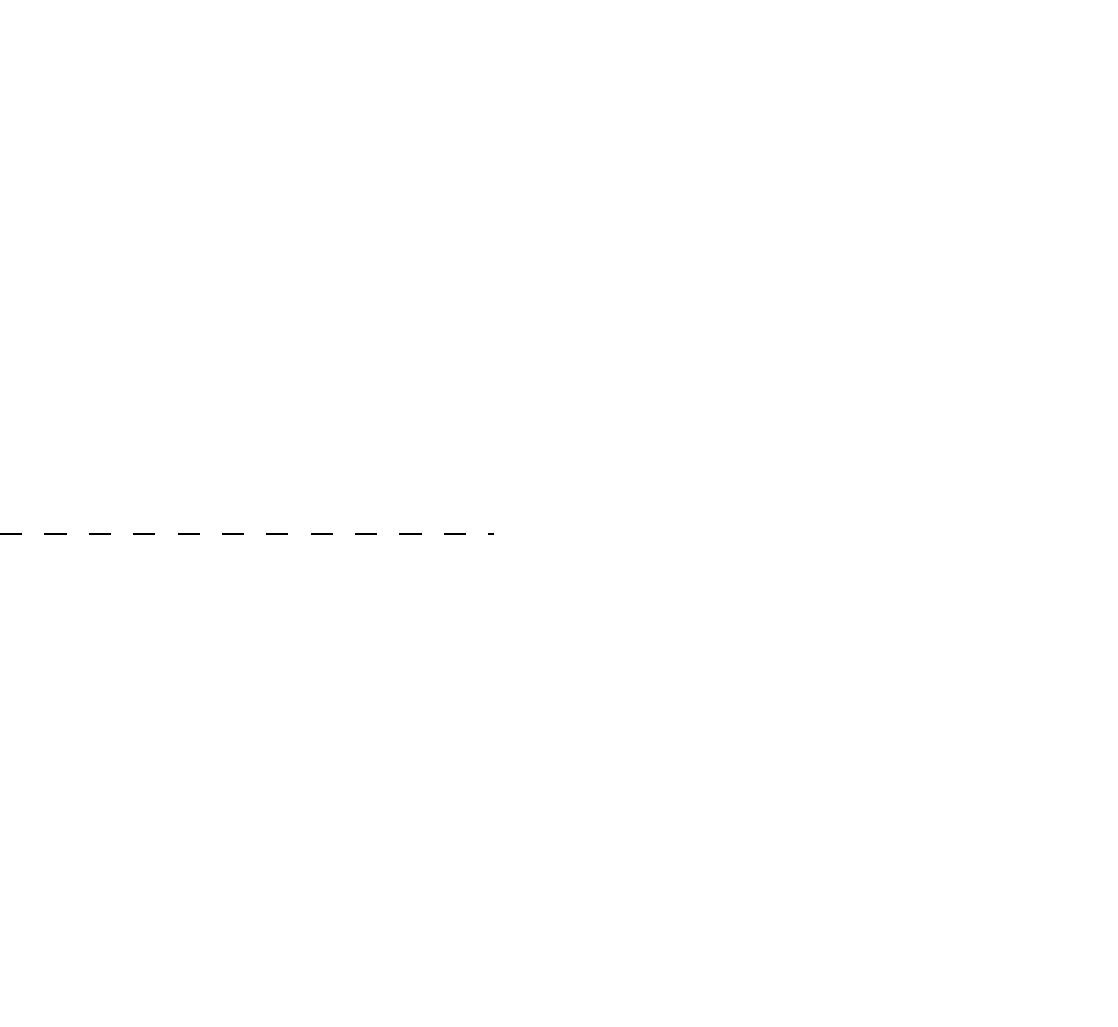 \label{fig:cro2}}
	\end{tiny}
	\caption{Examples of periodic orbits containing the origin.}
	\label{fig:crorbits}
\end{figure}

\begin{defi}
	A cycle is a closed curve formed by a finite set of pieces of orbits $\gamma_1, \dots, \gamma_n$ such that $\gamma_{2k}$ is a piece of sliding orbit, $\gamma_{2k+1}$ is a maximal regular orbit and the departing and arrival points of $\gamma_{2k+1}$ belong to $\gamma_{2k}$ and $\gamma_{2k+2}$, respectively. We define the period of the cycle as the sum of the times that are spent in each of the pieces of orbit $\gamma_i, \, i=1, \dots,n.$
\end{defi}

\begin{defi}
	We define a pseudo-cycle as the closure of a set of regular orbits $\gamma_1, \dots ,\gamma_n$ such that their edges, that is the arrival and departing points, of any $\gamma_i$ coincide with one of the edges of $\gamma_{i+1}$
	and one of the edges of $\gamma_{i+1}$ (and also between $\gamma_1$ and $\gamma_n$) forming a curve homeomorphic to $\mathbb{S}^1 = \R/ \mathbb{Z}$, in such a way that in some point coincide two departing or two arrival points.
\end{defi}

In Figures \ref{fig:crorbits} \subref{fig:cro1} and \subref{fig:cro} the curves $\gamma_1$ and $\gamma_2$ are examples of regular periodic orbits. As we will see in \autoref{sec:stability} and \ref{sec:cod1}, this kind of orbits do not appear in low codimension bifurcations. For example, the phase portrait sketched in \autoref{fig:crorbits} \subref{fig:cro1}, can happen when the origin is a saddle for $X$ (with non admissible eigenspaces, that is, the eigenspaces $V_{1,2} \subset \U^-$) and a focus for $Y$, which is a bifurcation of codimension at least four. The orbit $\gamma$ in \autoref{fig:crorbits}\subref{fig:cro}, gives us an example of a periodic cycle. These orbits can appear, for example, in the unfolding a codimension 2 fold-fold bifurcation.

In figure \ref{fig:cro2}, $\gamma_3$ is an example of pseudo-cycle. Observe that pseudo-cycles are not real closed orbits but they are preserved by $\s-$equivalences (see definitions \ref{def:s1} and \ref{def:s2}) and it is the unique type of recurrence containing the origin which appears in the unfoldings of some codimension one bifurcations.

\begin{figure}[!htbp]
	\centering
	\begin{tiny}
		\def\svgscale{0.35}
		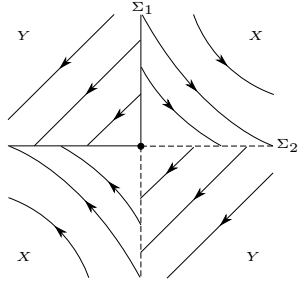 
	\end{tiny}
	\caption{$Z$ is transient, in this picture, the origin can be a saddle point for $X$ with non admissible eigenspaces and $Y$ is transverse to $\s$ at the origin.}
	\label{fig:trans1}
\end{figure} 

\begin{defi}We say that $Z \in \Omega$ is transient in $\U^\pm$ if for every $p \in \overline{\U^\pm}$ there exist $t_1(p), t_2(p) \in \R$ satisfying  $\varphi_{Z}(t_1(p);p) \in \s_1$, $\varphi_{Z}(t_2(p);p) \in \s_2$ and $\varphi_{Z}(t;p) \in \U^\pm$ for all $t \in [\min \{ t_1(p),t_2(p) \}, \linebreak \max \{ t_1(p),t_2(p) \} ]$. We say that $Z$ is transient if it is transient in $\U^+$ and $\U^-$.
\end{defi}

Let $Z \in \Omega$  transient. For each $p \in \s_1$ there exist a unique $t_X(p) \in \R$ such that $\varphi_{X}(p;t) \in \overline{\U^+}$ for all $t \in [\min \{ 0,t_X(p)\},\max \{ 0,t_X(p)\}]$ satisfying $\varphi_X(t_X(p);p) \in \s_2$. We have defined a diffeomorphism $$\phi_X: p \in \s_1 \mapsto \varphi_X(t_X(p);p) \in  \s_2.$$

By the Implicit Function Theorem, the function $t_X: p \in \s_1 \mapsto t_X(p) \in \R$ is a differentiable map. Analogously, we define  $$\phi_Y: p \in \s_2 \mapsto \varphi_Y(t_Y(p);p) \in  \s_1.$$

And finally we define the \textit{first return map of $Z$} by \begin{equation} \label{eq:poincare}
	\begin{array}{llll} 
		\phi_Z: & \s_2^- 	& \rightarrow 	& \s_2^- \\ 
		& p 	& \mapsto		&( \phi_X \circ \phi_Y)^2 (p)
	\end{array}
\end{equation}

\noindent which is clearly a diffeomorphism, since it is the restriction of a diffeomorphism to the cross section $\s^-_2$. 

\begin{figure}[!htbp]
	\centering
	\begin{tiny}
		\def\svgscale{0.35}
		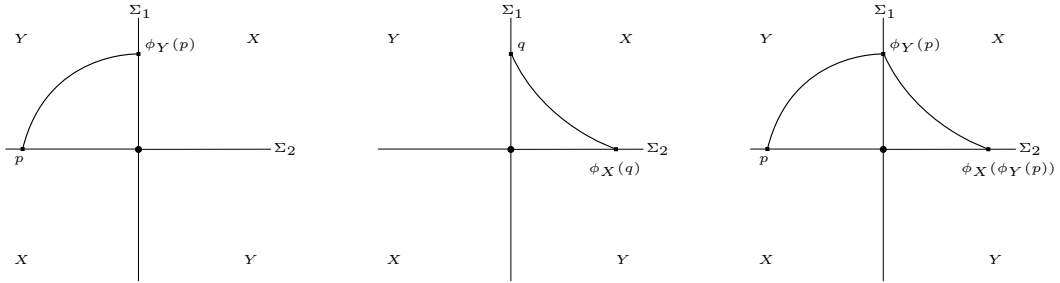 
	\end{tiny}
	\caption{The first return map associated to the vector field $Z$.}
	\label{fig:trans}
\end{figure}




\begin{rem}  
	In fact, the map $\phi_Z$ is an one dimensional map, since its second coordinate is always zero. Thus we will write $\phi_Z$ as a projection to the first coordinate of the first return map.
\end{rem}

Observe that if $Z$ is transient, the origin is always a fixed point for $\phi_Z$ since $t_X(\0)=t_Y(\0)=0$. Moreover, by unicity of solutions for flows in the plane, $\phi_Z$ is an increasing function and therefore, $\phi'_Z$ is always positive.

In our context, sometimes the first return map will not have a real dynamical meaning, since it is defined for any transient vector field $Z$. It can happen that a trajectory of $Z$ through a point $p\in \s_2^-$ do not reach again the cross section $\s_2^-$ neither in backward nor in forward time, see \autoref{fig:trans1}. But even in these cases the first return map will be important in order to detect the appearance of pseudo-cycles.

\begin{defi} 
	Let $Z \in \Omega$ be transient and let $\phi_Z$ be its first return map associated to $Z$ at the origin. Then the origin is ``geometrically stable'' if $0<\phi_Z'(0)<1$ and it is ``geometrically unstable'' if $\phi_Z'(0)>1$. 
	When $\s=\overline{\s^c}$ then the dynamics of $Z$ around the origin is similar to a focus, in this case we say that the origin is a ``focus''. Otherwise, $\0 \in \overline{\s^{e,s}_i \cap \s_j^c}$ the trajectories of an initial condition does not reach $\s_2^-$ again and the origin will be called ``geometric-focus''. 
\end{defi}

\bigskip


Now we are going to start with the definitions of local $\s-$structural stability and codimension $k$ bifurcations. 
It is well known that the set $\mathfrak{X}=\mathfrak{X}^r(\mathcal{U}),$ $\bar{\mathcal{U}}$ compact, of the germs of vector fields of class $\mathcal{C}^r, r \geq 1$ endowed with the $\mathcal{C}^r-$topology is a Banach space. Therefore, $\Omega=\mathfrak{X} \times \mathfrak{X}$ is also a Banach space. Consequently, $\Omega$ is a Banach manifold. 

In the sequel we will establish a relation between local $\s-$structural stability in our context with some special submanifolds of $\Omega.$

\begin{defi} \label{def:s1}
	Let $Z$ and $\tilde{Z} \in \Omega$, defined in $\U$ and $\tilde{\U}$ neighborhoods of the origin, with discontinuity sets $\s$ and $\tilde{\s}$, respectively. We say that $Z$ and $\tilde{Z}$ are locally $\s-$equivalent if there exist neighborhoods $\mathcal{U}_0, \tilde{\mathcal{U}_0}$ of the origin and an orientation preserving homeomorphism $h: \mathcal{U}_0 \rightarrow \tilde{\mathcal{U}}_0$ which maps trajectories of $Z$ in trajectories of $\tilde{Z}$ and sends $\s$ in $\tilde{\s}$.
\end{defi}

\begin{defi} \label{def:s2}
	We say that $Z \in \Omega$ is locally $\s-$structurally stable at the origin if there exists a neighborhood $\mathcal{V}_Z \subset \Omega$ such that if $\tilde{Z} \in \mathcal{V}_Z$ then is locally $\s-$equivalent to $Z$.
	
	Let $\Omega_0$ denote the set of all piecewise systems in $\Omega$ which are locally $\s-$struc\-tu\-ral\-ly stable. 
\end{defi}

When $Z$ is not locally $\s-$structurally stable at the origin, we say that $Z$ belongs to the bifurcation set $\Omega_1 = \Omega \setminus \Omega_0$. 

\begin{defi}
	Let $Z \in \Omega$. A $m-$parameter unfolding of $Z$ is a smooth map $\gamma: \delta=(\delta_1,\cdots,\delta_m) \in (-\delta_0,\delta_0)^m \mapsto Z_\delta \in \Omega$ with $\delta_0 \ll 1$, $m \geq 1$ and satisfying $\gamma(0)=Z_0=Z$. We usually denote an unfolding of $Z$ by $Z_\delta$. 
\end{defi}

\begin{defi}
	Let $Z, \tilde{Z} \in \Omega$. We say that two unfoldings of $Z_\delta$ and $\tilde{Z}_{\tilde{\delta}}$ are locally weak equivalent if there exists a homeomorphic change of parameters $\mu(\delta)$, such that, for each $\delta$ the vector fields $Z_\delta$ and $\tilde{Z}_{\mu(\delta)}$ are locally $\s-$equivalent. Moreover, given an unfolding $Z_\delta$ of $Z$ it is said to be a versal unfolding if every other unfolding $Z_{\alpha}$ of $Z$ is locally weak equivalent to $Z_\delta$.
\end{defi}

%

\begin{defi} \label{def:cod1}
	A piecewise smooth vector field $Z \in \Omega_1$ has a codimension one singularity at the origin if it is locally $\s-$structural stable in the induced topology of $\Omega_1$. That is, if there exists an open set $\mathcal{V}_Z \subset \Omega$ such that $\tilde{Z} \in \mathcal{V}_Z \cap \Omega_1$ then $\tilde{Z}$ is locally $\s-$equivalent to $Z$ and any $1-$parameter unfoldings of $Z$ and $\tilde{Z}$ are locally weak equivalent. We denote by $\Xi_1$ the set of all codimension one bifurcations in $\Omega.$
\end{defi}

One can define $\Xi_k$, the set of all $Z \in \Omega$ having a codimension $k$ bifurcation at the origin, recursively. Let  $\Omega_k = \Omega_{k-1}\setminus \Xi_{k-1}$ then $\Xi_k$ is the subset of $\Omega_k$ composed by the $\s-$structurally stable in $\Omega_k$. 

\begin{section}{\texorpdfstring{Local $\s-$}{s-}structural stability 
	} \label{sec:stability}
	
	The aim of this section is to describe the set $\Omega_0 \subset  \Omega$ of all the vector fields which are locally $\s-$structurally stable near the origin.  We use the definitions of local $\s-$equivalence and local $\s-$structural stability stated previously.  We will prove the following theorem: 
	
	\begin{theo}[$\s-$structural stability on $\Omega$] \label{theo:A} 
		Denote by $\Omega_0 \subset  \Omega$ the set of all the vector fields which are locally $\s-$structurally stable near the origin.
	 Let $Z \in \Omega$. Then $Z \in \Omega_0$ if, and only if, $Z$ satisfies one of the following conditions:
		\begin{enumerate}[A.]
			\item $X_i(\0).Y_i(\0)>0$, for $i=1,2$; 
			\item $X_i(\0).Y_i(\0)<0$, for $i=1,2$ and $\det{Z}(\0)=(X_1\cdot Y_2 - X_2 \cdot Y_1)(\0) \neq 0$;
			\item $X_i(\0).Y_i(\0)>0$, $X_j(\0).Y_j(\0)<0$ for $i=1,2, \, i \neq j$. In addition, when $Z$ is transient, it satisfies \linebreak$\displaystyle{\alpha_Z=\left( \frac{X_1 \cdot Y_2 (\0)}{X_2 \cdot Y_1(\0)} \right)^2 \neq 1}.$
		\end{enumerate}
		Moreover, the subset $\Omega_0$ is an open and dense in $\Omega$, therefore local $\s-$structural stability is a generic property in $\Omega$.
	\end{theo}
	
	We devote the rest of this section to prove this theorem.
	
	Lets consider $Z \in \Omega$, with $X$ and $Y$ transverse to $\s_1$ and $\s_2$ at the origin. It is clear that the transversality of the vector field with $\s$ at the origin is a necessary condition for local $\s-$structural stability of $Z$. In the sequel we will see that it is not a sufficient condition.

	Before we start the analysis of the behavior near the origin, we summarize some important facts about the sliding vector fields. By definition of the sliding vector fields expressed in \ref{eq:slidingdef}, as it was observed in \cite{SotoTei}, we have
	\begin{equation} \label{eq:slidingdetdef}
		Z_i^s(p)=h_i(p) \cdot \det{Z(p)}, \quad p \in \s_i,
	\end{equation}
	where $\det{Z(p)}= X_1(p)\cdot Y_2(p)-X_2(p)\cdot Y_1(p)$ and \begin{eqnarray}
		h_i(p)&=&[(-1)^{i-1} (X_i(p)-Y_i(p))]^{-1}, \label{eq:hidef}
	\end{eqnarray}
	Since each $Z_i^s$ is defined on sliding and escaping regions of $\s_i$,  it follows that $h_i$ does not vanish on these intervals since $X_i \cdot Y_i(p)<0$ if $p \in \s_i^{e,s}$. Then, we have the next proposition which proof follows directly.
	
	\begin{prop} \label{prop:pseudoeq}
		Let $Z=(X,Y) \in \Omega$. Then $p \in \s_i^{e,s}$ is a pseudo-equilibrium of $Z_i^s$, $i=1,2$, if and only if, $\det{Z(p)}=0$. In addition, $p \in \s_i^{e,s}$ is a hyperbolic pseudo-equilibrium to $Z_i^s$ provided $\displaystyle{\frac{\partial}{\partial x_j} \det{Z(p)} \neq 0}$, for $i=1,2$ and $i\neq j$.
	\end{prop}
	
	When both sliding vector fields are defined in a neighborhood of the origin, using \ref{eq:slidingdetdef} we obtain that $Z_1^s(0)=0$ if, and only if, $Z_2^s(0)=0$.
	
	The next proposition will be important when we construct the homeomorphisms in order to prove the local $\s-$equivalence between the $\s-$structural stable vector fields.

	\begin{prop} \label{prop:sliding} 
		Suppose that $p_0 \in \s_i^{s,e}$ is a regular point of the sliding vector field $Z_i^s$. Then $Z_i^s$ is locally conjugated to the constant vector field $\tilde{Z_i^s}(p)=  (-1)^{i-1} \sgn{X_i(p_0)} \cdot \sgn{\det{Z(p_0)}}$ by the identity map.
	\end{prop} 
	
	\begin{proof}
		A straightforward computation gives us that around $p_0$ we have sign $\sgn Z^s_i(p ) = \linebreak (-1)^{i-1} \sgn{X_i(p_0)}$. Since $Z^s_i$ and $\tilde Z^s_i$ are one dimensional vector fields, they are conjugated by the identity map.
	\end{proof}

	%
	
	\begin{corol} \label{corol:zsdirection}
		Let $Z=(X,Y) \in \Omega.$ Suppose that both sliding vector fields are defined around the origin, then $sgn(Z_1^s(0))=sgn(Z_2^s(0))$, if $X_1(\0)\cdot X_2(\0)<0$ or  $sgn(Z_1^s(\0))=-sgn(Z_2^s(\0))$, if $X_1(\0)\cdot X_2(\0)>0$.
	\end{corol}
	
	\begin{prop} \label{prop:regcos}
		Given $\Z \in \Omega$ suppose that the origin belongs to $\s_i^{e,s} \cap \s_j^c$ for $i,j=1,2$ and $i\neq j$. Then the origin is a regular point of $Z_i^s$.
	\end{prop}
	
	\begin{proof}
		Since we have $Z_i^s(0)=0$ if and only if, $\det{Z(p)}=0$, it is enough to prove that  $\det{Z(p)} \neq 0$. As $0 \in \s^{e,s}_i$, we have $X_i \cdot Y_i (\0)<0$ and $0 \in \s_j^c$ therefore $X_j \cdot Y_j (\0)>0$. Consequently, $sgn(X_i(\0).Y_j(\0))=sgn(-Y_i(\0).X_j(\0))$ and then $\det{Z(\0)}\neq 0.$
	\end{proof}
	
	From now on, our goal is to classify the local $\s-$structural stable behavior in $\Omega$. We give a normal form for each equivalence class and construct the respective local $\s-$equivalences between an arbitrary vector field and its corresponding normal form.
	
	Observe that being $X$ and $Y$ transverse to $\s$ at the origin, we have the following configurations for $\s$:
	
	\begin{enumerate} \renewcommand{\theenumi}{C\arabic{enumi}}
		\item $X_i \cdot Y_i(\0)>0$, for $i=1,2$ and then $\s=\s^c$;  \label{itm:C1}
		\item $X_i \cdot Y_i(\0)<0$, for $i=1,2$ and we have $\s=\s^s \cup \s^e$; \label{itm:C2}
		\item $X_i \cdot Y_i(\0)>0$ and $X_j \cdot Y_j(\0)<0$, then $\s_i=\s^c_i$ and $\s_j=\s^s_j \cup \s^e_j$ for $i,j=1,2$ and $i\neq j$; \label{itm:C3}
	\end{enumerate}
	
	Considering the continuous maps
	\begin{equation}  \label{eq:xii}
		\begin{array}{cccc}
			\xi_i: & \Omega & \rightarrow  	& \R \\
			& Z 			& \mapsto 		& X_i \cdot Y_i(\0)
		\end{array}
	\end{equation} \noindent it follows that conditions stated in items \ref{itm:C1} to \ref{itm:C3} are open. Then for each $Z \in \Omega$ satisfying conditions \ref{itm:C1} to \ref{itm:C3} there exists a neighborhood $\mathcal{V}_Z \subset \Omega$ such that $\sgn{\xi_i|_{\mathcal{V}_Z}}$ is constant. Moreover, conditions \ref{itm:C1} to \ref{itm:C3}  define a generic set, since its complement in $\Omega$ is $\xi_1^{-1}(0) \cup \xi_2^{-1}(0)$.  Nevertheless, even if these conditions are open, we will see that vector fields satisfying some of them are not structurally stable. We will analyze each case separately.

	\begin{definition}
	Let $\Omega_0^1 \subset  \Omega$ be the subset of all $Z \in \Omega$ satisfying condition \ref{itm:C1} and therefore, condition $A$ on \autoref{theo:A}. 
	\end{definition}
	
	The \autoref{prop:C1} gives the local $\s-$equivalence between any $Z \in \Omega_0^1$ and the corresponding normal form $\tilde{Z}$.
	
	Let $Z \in \Omega_0^1$ and $\mathcal{V}_Z \subset \Omega$ be a neighborhood of $Z$ such that $\sgn{\xi_i|_{\mathcal{V}_Z}}>0, \, i=1,2$. 
	By \autoref{prop:C1}, each $Z' \in \mathcal{V}_Z$ is locally $\s-$equivalent to $\tilde{Z}$. By transitivity it follows that $Z'$ is $\s-$equivalent to $Z$. Then every $Z \in \Omega_0^1$ is locally $\s-$structurally stable.

	\begin{prop} \label{prop:C1}
		Suppose that $Z \in \Omega_0^1$. Then $Z$ is locally $\s-$equivalent to the piecewise smooth system 
		
		\begin{equation} \label{eq:nfcrossing}
			\tilde{Z}(p)=\begin{cases}
				\tilde{X}(p)=(a,b), & \mbox{ if }\, p \in \U^+ \\
				\tilde{Y}(p)=(a,b), & \mbox{ if }\, p \in \U^-
			\end{cases}, 
		\end{equation} \noindent where $a=sgn(X_1(\0))$ and $b=sgn(X_2(\0))$.
		
		In other words, $Z$ is $\mathcal{C}^0-$equivalent to the continuous vector field $\tilde{Z}(p)=(a,b)$.
	\end{prop}
	
	\begin{proof} 
		We will present the construction for the case $X_1(\0), X_2 (\0)>0$. The other cases can be done analogously. In this case, the vector field $\tilde{Z}$ has the form \begin{equation}
			\tilde{Z}(x,y) = \left( \begin{array}{c} 1 \\ 1
			\end{array}
			\right) 
		\end{equation}

		Since $X$ and $Y$ are both transverse to $\s$ at the origin, there is a neighborhood $\U$ of the origin such that for each initial condition $p \in \U$ the trajectory of $Z$ through $p$ reaches the discontinuity $\s$ in finite time. 
		
		We are going to construct the homeomorphism piecewisely.
		
		\begin{figure}[!htbp]
			\centering
			\begin{tiny}
				\def\svgscale{0.35}
				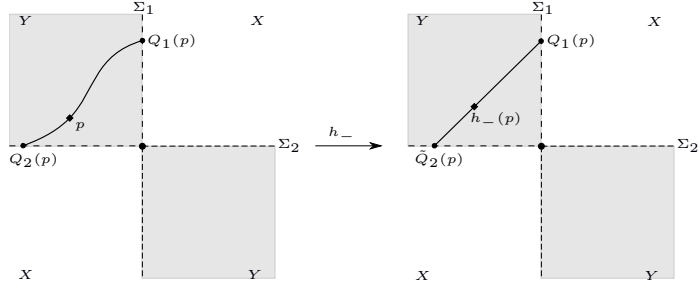
			\end{tiny}
			\caption{The map $h_-$ in $\U^-$ for $X_1 \cdot X_2 (\0) >0$.}
			\label{fig:homC11}
		\end{figure}
		
		\begin{figure}[!htbp]
			\centering
			\begin{tiny}
				\def\svgscale{0.35}
				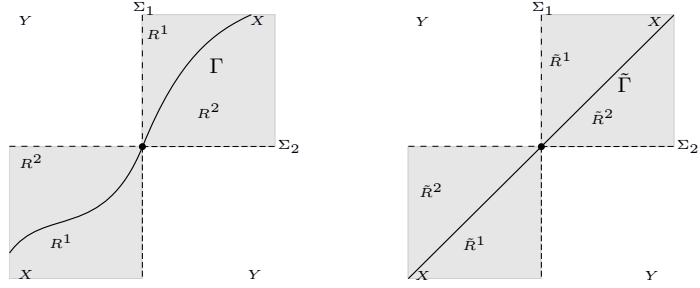
			\end{tiny}
			\caption{ The regions $R^i$ for $i=1,2$ of $\U$ and the respective cross sections $\Gamma$ and $\tilde{\Gamma}$ for $Z$ and $\tilde{Z}$ when $X_1 \cdot X_2 (\0)>0$.}
			\label{fig:homC12}
		\end{figure}
		
		Let $p \in \overline{\U^-}$. As $Y_1\cdot Y_2(\0)>0$, $Z$ is transient in $\U^-$. Then there exists a unique time $t_i(p) \in \R$ such that $Q_i(p)=\varphi_Y(p;t_i(p)) \in \s_i$ for $i=1,2$. 
		
		On the other hand, given $Q_1(p)=(0,q_1(p)) \in \s_1$, then $\tilde{Q}_2(p)=(-q_1(p),0)=\varphi_{\tilde{Y}}(Q_1(p);-q_1(p))$ belongs to $\s_2$.

		As illustrated in \autoref{fig:homC11}, we define the homeomorphism $h^-$ on $\overline{\U^-}$ as 
		\begin{equation*}
			h^-(p) = \varphi_{\tilde{Y}}\left( Q_1(p), \sigma(p) \right) 
		\end{equation*} where $\sigma(p)=- \displaystyle{\frac{q_1 \cdot t_1}{t_2 - t_1} (p)}$. \newline
		
		Observe that $h^-|_{\s_1} = Id$, $h^-(p) = (-q_1(p),0)$ if $p \in \s_2$ and $h^-(\0)= \0$. Moreover, the map $h^-$ is an homeomorphism in $\U^-$.
		
		Consider now the cross section $\Gamma=\{ \varphi_X(0,t), \, t \in \R \} \cap \U$ and define the regions \begin{align*}
			R_1= \{ p \in \U^+_+ : p \mbox{ is above } \Gamma \} \cup \{ p \in \U^+_- : p \mbox{ is below } \Gamma\} \cup \Gamma, \\
			R_2= \{ p \in \U^+_+ : p \mbox{ is below } \Gamma \} \cup \{ p \in \U^+_- : p \mbox{ is above } \Gamma\} \cup \Gamma.
		\end{align*}

		Since $\s=\s^c$ for each $p \in \overline{R_i}$ there exists a unique $t_i(p) \in \R$ such that $Q_i(p)=\varphi_X(p,t(p)) \in \s_i$, $i=1,2$.

		Define on each region $\overline{R_i}$ the homeomorphisms \begin{eqnarray*}
			h_1^+(p)&=& \varphi_{\tilde{X}}(Q_1(p);-t_i(p)), \; p \in R_1, \\
			h_2^+(p)&=& \varphi_{\tilde{X}}(h_-(Q_2(p));-t_i(p)) , \; p \in R_2.
		\end{eqnarray*}
		
		Observe that if $p \in \Gamma$ then $h_1^+ (p) = \varphi_{\tilde{X}}(0,-t(p)) = h_2^+(p)$. Therefore, the map $h^+$, defined as $h^+(p)=  h^+_1(p)$, if $p \in R_1$ and $h^+(p)=  h^+_2(p)$, if $p \in R_2$ is a homeomorphism in $\U^+$.
		
		Moreover, the maps $h^+$ and $h^-$ agree on $\s$. If $p \in \s_i,$ then $Q_i(p)=p$ and $t_i(p)=0$, thus $h^-(p)=h^+_{1,2}$. Therefore the map $h$ defined as follow is an homeomorphism. \begin{equation*}
			h(p)= \begin{cases} 
				h^-(p)=\varphi_{\tilde{Y}}( \varphi_Y(p,t_1(p)),\sigma(p)), & \, p \in \overline{\U^-}, \\
				h^+_1(p)=\varphi_{\tilde{X}}(\varphi_X(p,t_1(p)),-t_2(p)), & \, p \in \overline{R_1}, \\
				h^+_2(p)=\varphi_{\tilde{X}}\left(\varphi_{\tilde{Y}}\left(\varphi_Y\left(\varphi_X(p;t_2(p));t_1 \left(\varphi_X(p;t_2(p)) \right)\right),\sigma(\varphi_X(p;t_2(p)))\right);-t_2(p)\right), & \, p \in \overline{R_2} \end{cases} .
		\end{equation*}
		
		\begin{figure}[!htbp]
			\centering
			\begin{tiny}
				\def\svgscale{0.4}
				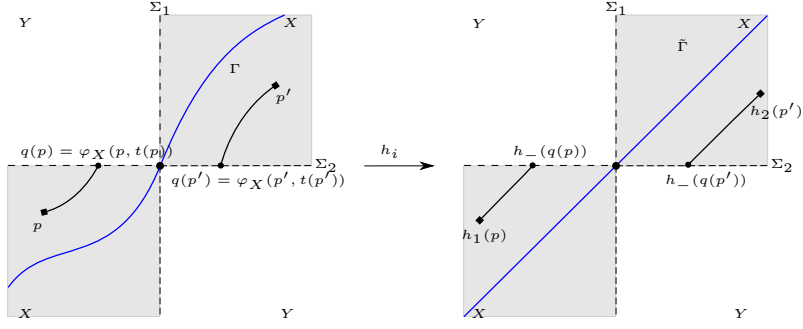
			\end{tiny}
			\caption{The homeomorphism $h_i$ define on regions $R^i$ for $X_1 \cdot X_2 (\0)>0 $}
			\label{fig:homC13}
		\end{figure}
		
		Regarded the way that $h$ has been constructed it is clear that $h$ carries the trajectories of $Z$ to trajectories of $\tilde{Z}$ preserving the orientation. Moreover, if $p =(p_1,p_2) \in \U$, a straightforward calculation shows that the map 
		
		\begin{equation*}
			g(p)= \begin{cases} 
				\varphi_{Y}((0,p_2-p_1));\tau(p)), & \, p \in \overline{\tilde{\U}^-}, \\
				\varphi_{X}((0,p_2-p_1);p_1), & \, p \in \overline{\tilde{R}_1}, \\
				\varphi_X(\varphi_{Y}((0,p_2-p_1);-t_2(p_2-p_1);p_2) & \, p \in \overline{\tilde{R}_2},\\
			\end{cases}
		\end{equation*} \noindent where $\displaystyle \tau (p)=\frac{p_1 \cdot t_2(0,p_2-p_1)}{p_1-p_2}$ is the inverse of $h$. Thus $h$ is a homeomorphism.
		
		The case $X_1(\0),X_2(\0)<0$ is analogous. The only difference is that we must consider the  cross section $\Gamma=\{ \varphi_{Y}(t,0) \} : t \in \R \} \cap \U$, define the regions $R_i$ and then proceed in the same way as in the case $X_1(\0), X_2 (\0)>0$.
	\end{proof}

	Opposed to the case \ref{itm:C1}, if $Z \in \Omega$ satisfies \ref{itm:C2} one can not automatically conclude that $Z$ is locally $\s-$structurally stable even if \ref{itm:C2} is an open condition. This occurs because one can have that $Z$ satisfies \ref{itm:C2} and $\det Z(\0)=0$ and then the origin is a pseudo-equilibrium for both $Z_i$. Lets take for instance \begin{equation}
		Z_{\alpha}(x,y)= \begin{cases}
			X_{\alpha}(x,y)&=(1-\alpha +x,1) \\
			Y(x,y)&= (-1+y,-1)
		\end{cases}.
	\end{equation}

	In this case, the sliding vector field is defined in $\s_1$ and $\s_2$. However, for $\alpha=0$ the vector field $Z_0$ has just one pseudo-equilibrium at the origin, while when $\alpha \neq 0$ the vector field $Z_\alpha$ has two pseudo-equilibria in $\s_1$ and $\s_2$  near the origin. Therefore, one can not establish a local $\s-$equivalence between $Z_0$ and $Z_\alpha$ for $\alpha \neq 0.$

	Since $\det: Z \in \Omega \mapsto \det{Z(\0)} \in \R$ is a continuous function, once $Z$ satisfies $\det Z(\0) \neq 0$, we obtain a neighborhood $\mathcal{W}_Z \subset \mathcal{V}_Z$ such that $\sgn{\det|_{\mathcal{W}_Z}}$ is constant. Then $Z' \in \mathcal{W}_Z$ satisfies condition \ref{itm:C2} and $\det{Z'(\0)} \neq 0$.
	
	\begin{definition} Let $\Omega_0^2$ be the set of all $Z \in \Omega$ satisfying \ref{itm:C2} and $\det Z(\0) \neq 0$ and therefore condition $B$ of \autoref{theo:A}.
	\end{definition}	
		
Then from the argument exposed above and the next proposition we conclude that $Z \in \Omega_0^2$ is local $\s-$structurally stable. 
	
	\begin{rem} Even if the formula stated for the normal form of $\tilde{Z} \in \Omega^2_0$ in system \ref{sys:C2nf} is cumbersome, it is a good way to write all the normal forms in a concise way. Substituting the values for $a,b$ and $c$ indicated in \autoref{prop:C2} the expression of $\tilde{Z}$ becomes very simple.   
	\end{rem}

	\begin{prop} \label{prop:C2}
		
		Let $Z \in \Omega_0^2$. Then $Z$ is locally $\s-$equivalent to the piecewise smooth system 
		\begin{equation} \label{sys:C2nf}
			\tilde{Z}(p)=
			\begin{cases}
				((\delta_{-1(ab)}\delta_{1c}+1) \cdot a,-(\delta_{-1(ab)}\delta_{-1c}-ab)\cdot a), & \mbox{ if } p \in \U^+ \\
				(-(\delta_{1(ab)}\delta_{1c} + 1) \cdot a ,-(\delta_{1(ab)}\delta_{-1c} +ab)  \cdot a ), & \mbox{ if } p \in \U^-
			\end{cases}, 
		\end{equation} where $a=sgn (X_1(\0))$, $b=sgn (X_2 (\0))$, $c=sgn(\det Z(\0))$ and $\delta_{rs}$ is the Kronecker function.
	\end{prop}
	
	\begin{proof}
		Observe that $\tilde Z$ also satisfies condition \ref{itm:C2}. Moreover, substituting the value of $ab$ in the formula \ref{sys:C2nf}, we obtain $\det{\tilde Z}(\0)=\delta_{1c} - \delta_{-1c}$. Thus, $\sgn {\det{\tilde Z}(\0)}= \sgn{\det{ Z}(\0)} \neq 0.$ 
		
		As $Z_i^s$ and $\tilde{Z}_i^s$ are one dimensional vector fields with the same sign, the identity is an equivalence between them. Therefore, the same construction of \autoref{prop:C1} can be applied in this case.

	\end{proof}
	
	The last case to be studied is when $Z$ belongs to $\Omega^3_0$, that is, the set of all $Z \in \Omega$ satisfying \ref{itm:C3}. Observe that in the previous cases there were no meaningful differences on the dynamics of $Z$ depending of $\sgn{X_1 \cdot X_2(\0)}$, since one case is just the reflection of the other. This is not true when we are considering $Z \in \Omega_0^3$.
	
	For this reason, we define,
	
	\begin{definition} 
		Let $\Omega_0^{3}$ the set of all $Z \in \Omega$ satisfying condition $C3$ and  
		$$\Omega_0^{3,1} = \{ Z \in \Omega_0^3 : \, X_1 \cdot X_2 (\0) >0 \}.$$
		\end{definition}
		
		A standard argument and \autoref{prop:C31} show that $Z \in \Omega_0^{3,1}$ is always local $\s-$structurally stable. Observe that this is not the case when $Z \in \Omega_0^3 \setminus \Omega_0^{3,1}$. In fact, if $X_1 \cdot X_2 (\0)<0$, the vector field $X$ is transient. Moreover, condition \ref{itm:C3} gives straightforwardly that the vector field $Y$ is also transient. Consequently, since $Z$ is transient, one needs to analyze the first return map defined in \ref{eq:poincare} in order to avoid a non hyperbolic fixed point of $\phi_Z$ at the origin. This situation would lead to a higher codimension bifurcation.
	
	\begin{prop}\label{prop:C31}
		Suppose $Z \in \Omega^{3,1}_0$. Then $Z$ is locally $\s-$equivalent around at the origin to $$\tilde{Z}(p)=\begin{cases}
		\tilde{X}(p)=(a,a), & \mbox{ if } p \in \U^+ \\
		\tilde{Y}(p)=(b,-b), & \mbox{ if } p \in \U^-
		\end{cases},$$ where $a=sgn(X_1(\0))$ and $b=sgn(Y_1 (\0))$.
	\end{prop}

	\begin{figure}[!htbp]
		\centering
		\begin{tiny}
			\def\svgscale{0.35} 
			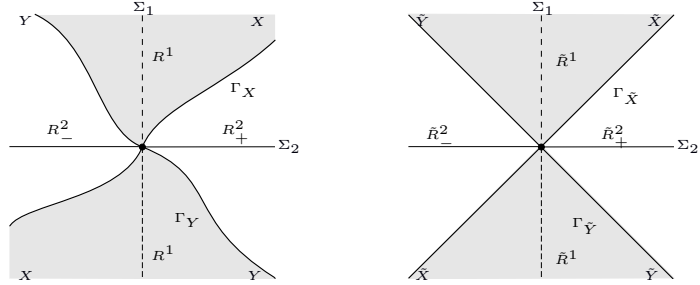
		\end{tiny}
		\caption{The cross sections $\Gamma_{X,Y}$ and regions $R^i$, $i=1,2$ when $a=1$.} 
		\label{fig:homC31}
	\end{figure}
	
	\begin{proof}
		Fix $a=b=1$. Other cases can be done similarly.
		
		In this case we have $\s_1=\overline{\s^c}$ and $\s_2 = \overline{\s^e} \cup \overline{\s^e}$, by \autoref{prop:regcos} the origin is a regular point of $Z^s_2$ and by \autoref{prop:sliding}, $Z^s_2$ is locally conjugated to the constant vector field $\tilde{Z}^s_2(p)=1$. Moreover, by \autoref{prop:sliding} there exists a homeomorphism $h^*$ with $h^*(0)=0$ which gives the $\mathcal{C}^0-$equivalence between these vector fields in a neighborhood of the origin.

		The trajectories of $X$ and $Y$ through the origin are both admissible and intersect $\s$ transversely at this point and the same occurs for $\tilde{Z}$.
		
		
		Consider the following cross sections of $\s$ given by $\Gamma_X  =  \{ \varphi_X(t,0) : t \in \R \} \cap \U$ and $  \Gamma_Y =  \{ \varphi_Y(t,0) : t \in \R \} \cap \U$. Analogously, we define the cross sections $\Gamma_{\tilde{X}}$ and $\Gamma_{\tilde{Y}}$ for $\tilde{Z}$. See  \autoref{fig:homC31}.
		
		Let $R_i \subset \U$ be the region between $\Gamma_X$ and $\Gamma_Y$ which contains $\s_i$ for $i=1,2$. In addition, decompose $R_2$ into two regions given by $R_2^\pm = R_2 \cap \U^\pm$. Proceeding in the same way as above, we define $\tilde{R}_1$ and $\tilde{R}_2^\pm$ contained in $\tilde{\U}$.

		For each $p \in \overline{R_1}$ there exists a unique $t_1(p) \in \R$ such that $q(p)=\varphi_Z(p,t_1(p)) \in \s_1$. Then set $$h(p)=\varphi_{\tilde{Z}}(\varphi_Z(p,t_1(p)),-t_1(p)).$$
		
		If $p\in \overline{R_2}$ then $t_2(p) \in \R$ is the unique time such that $q(p)=\varphi_X(p,t_2(p)) \in \s_2$ if $p \in \overline{R_2^+}$ and $q(p)=\varphi_Y(p,t_2(p)) \in \s_2$ if $p \in \overline{R_2^-}$. In both cases, $h^*(q(p))$ belongs to $\tilde{\s}_2$ and if $p \in \overline{R_2^+}$ we define $h(p)=\varphi_{\tilde{X}}(h^*(q(p)),-t_2(p))$ and in case $p \in \overline{R_2^-}$ we set $h(p)=\varphi_{\tilde{Y}}(h^*(q(p)),-t_2(p)).$

		The three functions defined above are homeomorphisms and due to the way they were constructed they also agree on the intersections, hence the map $$h(p)=\begin{cases}
		\varphi_{\tilde{Z}}(\varphi_Z(p,t_1(p)),-t_1(p)), & p \in \overline{R_1}, \\
		\varphi_{\tilde{X}}(h^*(\varphi_X(p,t_2(p)))),-t_2(p)), & p \in \overline{R_2^+}, \\
		\varphi_{\tilde{Y}}(h^*(\varphi_Y(p,t_2(p)))),-t_2(p)), & p \in \overline{R_2^-}. \\
		\end{cases}$$ is a homeomorphism which carries trajectories of $Z$ into trajectories of $\tilde{Z}$ preserving the orientation and so they are locally $\s-$equivalent.
	\end{proof}

	Finally suppose that $Z \in \Omega_0^3 \setminus \Omega^{3,1}_0$, that is, $Z$ satisfies $X_1 \cdot X_2(\0)<0$. Under these hypothesis, $Z$ is transient. Thus we need to understand what happens with the first return map $\phi_Z$ defined in \ref{eq:poincare}. 
	
	\begin{figure}[!htbp]
		\centering
		\begin{tiny}
			\def\svgscale{0.35} 
			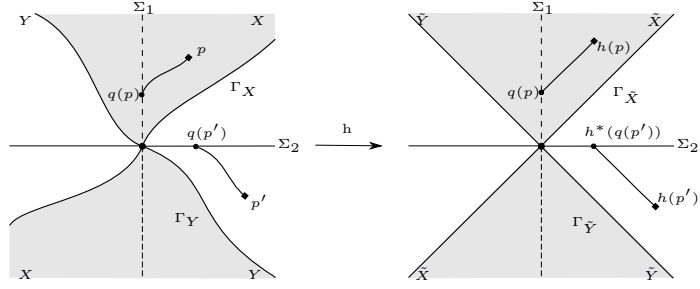
		\end{tiny}
		\caption{The homeomorphism $h$ for $a=1$. In this case $h$ is defined independently in each region $R^i$ in a way to agree on the cross sections.} 
		\label{fig:homC32}
	\end{figure}
	
	It is important to notice that since there are sliding and escaping regions in $\s$, given $p \in \s_2^-$ its $Z$ trajectory does not reach $\s_2^-$ again. Thus there are no regular periodic orbits for this case. Even though, one can exist pseudo-cycles which are preserved by $\s-$equivalences.  
	
	In general, it is not easy to compute explicitly the first return map. The next proposition gives an approximation of its expression when $Z$ is transient and transverse to $\s$ around the origin.
	
	\begin{prop} \label{prop:poinctrans}
		Let $Z \in \Omega_0^3$ satisfying $X_1 \cdot X_2(\0)<0$. Then the first return map is given by $$\phi_Z(x)= \alpha_Z^2 x + \mathcal{O}(x^2)$$ \noindent with \begin{equation} \label{alphaZ} \alpha_Z=\frac{X_1 \cdot Y_2 (\0)}{X_2 \cdot Y_1(\0)}.\end{equation}
	\end{prop}
	
	\begin{proof}
		Near the origin, the trajectories of $Y$ through a point $(x,0) \in \s_2$ can be written as \begin{equation} \label{2.9}
			\varphi_Y ((x,0),t)= (x,0) + (Y_1(x,0), Y_2(x,0))t + \mathcal{O}(t^2).
		\end{equation}
		From this equation we obtain that the time to arrive $\s_1$ is $t=\displaystyle{\frac{x}{Y_1(\0)}+\mathcal{O}(x^2)}.$
		Then, using again \ref{2.9}, we have  
		$$
		\phi_Y(x)  =  \displaystyle{-\frac{Y_2(\0)}{Y_1(\0)}} x + \mathcal{O}(x^2).
		$$
Analogously, for all $(y,0)$ near the origin, we obtain
		$$\phi_X(y) =  \displaystyle{-\frac{X_1(\0)}{X_2(\0)}} y + \mathcal{O}(y^2).  
		$$
By composing twice $\phi_X$ and $\phi_Y$ we obtain the desired map.
	\end{proof}
	
	Observe that the constant $\alpha_Z$ is always negative in this case. Then by \autoref{prop:poinctrans}, the origin is a hyperbolic fixed point for $\phi_Z$ if and only if, 
	\begin{equation} \label{form:gamma}
		\gamma_Z=X_1 \cdot Y_2 (\0)+X_2 \cdot Y_1(\0) \neq 0
	\end{equation}
In addition, the origin will be stable if $\alpha_Z+1 > 0$ and unstable if $\alpha_Z + 1 < 0.$

\begin{definition} 
	Let $\Omega_0^{3,2}$ be the subset of $\Omega_0^3$ such that $X_1 \cdot X_2 (\0)<0$ and $\alpha_Z \neq -1$. 
	\end{definition}
It is clear that given $Z \in \Omega_0^{3,2}$ there exists a neighborhood $\mathcal{W}_Z$ of $Z$ such that $Z' \in \mathcal{W}_Z$ then $Z' \in \Omega_0^{3,2}$ with $\sgn{\alpha_Z+1} = \sgn{\alpha_{Z'}+1}$. As an easy consequence of the next proposition we have that $Z$ and $Z'$ are locally $\s-$equivalent and thus $Z$ is locally $\s-$structurally stable.   
	
	\begin{rem} \label{rem:hyppoinc} Observe that the origin is attractive of $\phi_Z$ if $\left| \alpha_Z \right|<1$, this is equivalent to $\gamma_Z=\left|X_1 \cdot Y_2 (\0)\right| + \left|X_2 \cdot Y_1 (\0)\right| = X_1 \cdot Y_2 (\0)+X_2 \cdot Y_1(\0) <0$. Analogously, the origin is repelling if $\gamma_Z >0.$
	\end{rem}
	
	\begin{prop} \label{prop:C32}
		Suppose $Z \in \Omega_0^{3,2}$. Then $Z$ is locally $\s$-equivalent to $$\tilde{Z}(p)= \begin{cases}
		(a,-a), & \mbox{ if } p \in \U^+, \\
		(b(1+\delta_{1c}),b(1+\delta_{-1c})), & \mbox{ if } p \in \U^-,
		\end{cases} $$ \noindent with $a=sgn(X_1(\0))$, $b=sgn(Y_1(\0))$ and $c=\sgn{\alpha_Z+1}$.
	\end{prop}
	
	\begin{proof}
		We are going to fix $a, b, c = 1$, the other cases can be treated analogously. First of all, observe that $\alpha_{\tilde{Z}}=-\frac{1}{2}$, thus the origin is an stable hyperbolic fixed point to $\phi_{\tilde{Z}}$. By the Grobman-Hartman Theorem there exists a homeomorphism $h^*$ defined in a neighborhood of the origin such that $\phi_Z \circ h^* = h^* \circ \phi_{\tilde{Z}}$.
		
		Moreover, by \autoref{prop:sliding} and \autoref{prop:regcos} we have that $Z_2^s$ and $\tilde{Z}_2^s$ are $\mathcal{C}^0-$e\-qui\-va\-lent in a suitable neighborhood of the origin.

		Let consider $R^+ = \{ (x,y) \in \U : y > 0 \}$ and $R^- = \{ (x,y) \in \U : y < 0 \}$.  Since $a, b=1$ then $\s_1 = \overline{\s^c}$ and so the trajectories of $Z$ through a point of $\U^+$ or $\U^-$ is an appropriate concatenation of the trajectories of $X$ and $Y$. The same is valid for $\tilde{Z}$.

		For each $(q,0) \in \s_2^-$ there exist unique times $t_1(q),t_2(q) \in \R$ such that $(0,q_1^+)=\varphi_Y((q,0);t_1(q)) \in \s\, _1^+$ and $(q_2^+,0)=\varphi_X((0,q_1^+),t_2(q)-t_1(q)) \in \s_2^+$. This defines a continuous curve (see \autoref{fig:homC321}) in $\overline{R^+}$ \begin{equation} \eta^+_q (t) =  \begin{cases}
				\varphi_Y((q,0);t), & t \in I_q^1= [0,t_1(q)], \\
				\varphi_X(\varphi_Y((q,0);t_2(q)-t_1(q)),t), & t \in I^2_q = [t_1(q),t_2(q) ].
			\end{cases}
		\end{equation}
		
		\begin{figure}[!htbp]
			\centering
			\begin{tiny}
				\def\svgscale{0.4} 
				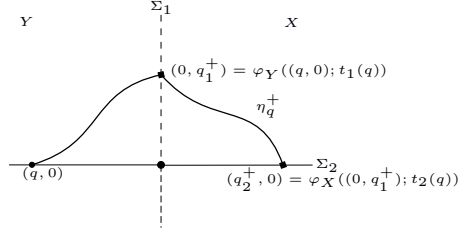
			\end{tiny}
			\caption{The trajectory of every point $q \in \s_2^-$ can be seen as a continuous curve $\eta^+_q$.} 
			\label{fig:homC321}
		\end{figure}
		
		Analogously we define the curve $\eta^-_q$ in $\overline{R^-}:$ \begin{equation} \eta^-_q (s) =  \begin{cases}
				\varphi_Y((q,0);s), & s \in J^1_q = [0,s_1(q)], \\
				\varphi_X(\varphi_Y((q;0);s_2(q)-s_1(q));s), & s \in J^2_q = [s_1(q),s_2(q)],
			\end{cases}
		\end{equation} \noindent where $s_1(q), s_2(q) \in \R$ satisfy $(0,q_1^-)=\varphi_Y((q,0);s_1(q)) \in \s_1^-$,and $(q_2^-,0)=\varphi_Y((0,q_1^-);s_2(q)-s_1(q)) \in \s_2^+.$
		
		For each $p \in \overline{R^+}$ there exists a unique point $(q(p),0) \in \s_2^-$ and a unique $t(p)$ such that $p = \eta_{q(p)}^+(t(p))$ for some $t(p) \in I_{q(p)}=I^1_{q(p)} \cup I^2_{q(p)}.$ 
		
		Consider $h^*(q(p)) \in \tilde{\s}_2^-$ and let $$\tilde{\eta}^+_{h^*(q(p))}: \tilde{I}_{h^*(q(p))} \rightarrow \overline{\tilde{R}^+}$$ be its $\tilde{Z}$ trajectory. Where  $$\tilde{I}_{h^*(q(p))}=[0,-\frac{1}{2}h^*(q(p))] \cup [-\frac{1}{2}h^*(q(p)),-h^*(q(p))].$$
		
In order to get an equivalence that preserves $\s$, we will make a reparametrization of time $\sigma^+_p$ which preserves the subintervals of  $I_{q(p)}$ and $\tilde{I}_{h^*(q(p))}$. Then define the homeomorphism $h^+$ on $\overline{R^+}$ by \begin{equation} 
			h^+(p)= \tilde{\eta}^+_{h^*(q(p))}(\sigma_p^+(t(p))).
		\end{equation}
		
		Let $r(p)$ be the intersection between the trajectory of $p$ in $R^-$ with $\s_2^-$. Proceeding in the same way we define the homeomorphism $h^-$ in $\overline{R^-}$ by \begin{equation} 
			h^-(p)= \tilde{\eta}^-_{h^*(r(p))}(\sigma_p^-(s(p))).
		\end{equation}
		
		\begin{figure}[!htbp]
			\centering
			\begin{tiny}
				\def\svgscale{0.4} 
				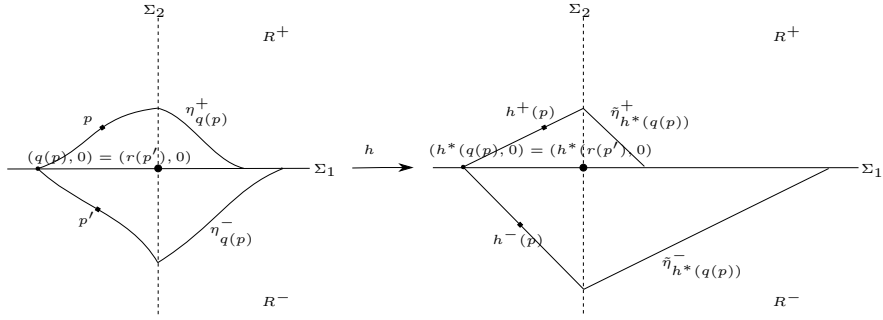
			\end{tiny}
			\caption{The homeomorphism which gives the equivalence between $Z \in \Omega^{3,2}_0$ and its normal form for $a,b,c=1$.} 
			\label{fig:homC322}
		\end{figure}
		
		Given $p \in \s_2$, then if $p \in \overline{\s_2^-}$ then $h^+(p)=h^*(p)=h^-(p)$ since both maps coincide with $h^*$ in $\s_2^-$. On the other hand, given $p \in \overline{\s_2^+}$ the trajectory of $p$ in $\overline{R^\pm}$ intersects $\s_2^-$ at the point $q(p)$ and $r(p)$, respectively. Observe that $r(p)=\phi_Z(q(p))$ and since $h^* \circ \phi_Z = \phi_{\tilde{Z}} \circ h^*$ we have that $\phi_{\tilde{Z}}(h^*(q(p)))=h^*(r(p))$, therefore $h^*(q(p))$ and $h^*(r(p))$ belong to the same trajectory, then $h^+(p)=h^-(p)$. 
		
		As the maps $h^+$ and $h^-$ agree in $\s_2$, then we can define $h: \U \rightarrow \tilde{\U}$ given by $$ h(p) = \begin{cases}
		\tilde{\eta}^+_{h^*(q(p))}(\sigma_p^+(t(p))), &  p \in \overline{R^+}, \\
		\tilde{\eta}^-_{h^*(r(p))}(\sigma_p^-(s(p))), &  p \in \overline{R^-}.
		\end{cases}
		$$

		Since all the maps involved are continuous and bijective, we conclude that $h$ is also a bijective and continuous function. The inverse of $h$ can be constructed in the same way and is also continuous, then $h$ is a homeomorphism.

		Due to the way that $h$ was constructed it is clear that $h$ maps all trajectories of $Z$ in trajectories of $\tilde{Z}$, including the trajectories of the sliding vector field, since $h$ preserves the order on $\s_2$.  
	\end{proof}
	
	Let us observe that the set $\Omega_0^{3,1} \cup \Omega_0^{3,2}$ consists of the piecewise vector fields satisfying condition $C$ of \autoref{theo:A}. 
	
	We are now in conditions to present a characterization of $\Omega_0$, the set of all locally $\s-$struc\-tu\-ral\-ly stable systems in $\Omega$. Set $$\Omega_0' =\Omega_0^1 \cup \Omega_0^2 \cup \Omega_0^{31} \cup \Omega_0^{32} \subset \Omega,$$ during this section we had shown that $\Omega_0' \subset \Omega_0.$
	On the other hand, if $Z \in \Omega \setminus \Omega_0'$ then $Z$ satisfies at least one of the following items: \begin{enumerate}[(i)]
		\item The map $\xi_i(Z)=0$ for $i=1$ or $2$; \label{itm:i}
		\item $Z \in \Omega_0^2$ but $\det{Z(\0)}=0$; \label{itm:ii}
		\item $Z \in \Omega_0^{3,2}$ with $\gamma_Z=0$ (see \ref{form:gamma}), or equivalently, $\alpha_Z=-1$ (see \ref{alphaZ}). \label{itm:iii}
	\end{enumerate}
	
	If $Z \in \Omega \setminus \Omega_0$ satisfies (\ref{itm:i}), then $Z$ is tangent to $\s$ at the origin. In this case, the family $$Z_n = \begin{cases}
	X + (\frac{1}{n},\frac{1}{n}) & \, p \in \U^+ \\
	Y + (\frac{1}{n},\frac{1}{n}) & \, p \in \U^-
	\end{cases}$$ \noindent converges to $Z$ when $n \rightarrow \infty$ and it is transverse to $\s$ at the origin for all $n \in \mathbb{N}$, then $Z$ and $Z_n$ can not be locally $\s-$equivalent. Therefore, $Z$ is not local $\s-$structurally stable.
	
	In case $Z$ satisfies (\ref{itm:ii}) or  (\ref{itm:iii}) the family $$Z_n = \begin{cases}
	X + (\frac{1}{n},0) & \, p \in \U^+ \\
	Y  & \, p \in \U^-
	\end{cases}$$ \noindent also converges to $Z$. Moreover, observe that $\det{Z_n(\0)} \neq 0$ and $\gamma_{Z_n} \neq 0$, thus one can not establish a $\s-$equivalence between $Z$ and $Z_n$. Concluding that if $Z \in \Omega$ does not belong to $\Omega_0'$ then $Z$ is not structurally stable, thus $\Omega_0'=\Omega_0$.
	
	In addition, the above argument also shows that $\Omega_0$ is dense in $\Omega$ since for every neighborhood of $Z \in \Omega$ there exists a sequence $Z_n \in \Omega_0$ such that $Z_n \rightarrow Z$ when $n \rightarrow \infty$.  Joining the results we stated by now, we have proved \autoref{theo:A}.
	
\end{section}

\begin{section}{Codimension one generic bifurcations} \label{sec:cod1}
	
	Once we have classified the generic behavior of $Z \in \Omega_0$, we will now investigate what happens in the bifurcation set $\Omega_1 = \Omega \setminus \Omega_0$.
	
	The goal of this section is to classify the codimension one generic bifurcation set, which will be called $\Xi_1 \subset \Omega_1$, that is, classify the set of the structurally stable piecewise smooth systems in $\Omega_1$ endowed with the induced topology of $\Omega$. In order to $Z$ belong to $\Xi_1$ it is necessary to break at most one of the conditions stated in \autoref{theo:A}. We must consider the following three groups:
	
	\begin{itemize}
		\item $X$ or $Y$ has a tangency point at the origin, that is, $X_i(\0)=0$ or $Y_i(\0)=0$ for $i=1,2$; \label{itm:T}
		
		\item $Z$ satisfying condition \ref{itm:C2} and $\det{Z(\0)}=0$, that is, the origin is a pseudo-equilibrium for the sliding vector fields $Z^s_i$, $i=1,2$; \label{itm:P}
		
		\item $Z$ satisfying \ref{itm:C3}, $X_1 \cdot X_2 (\0)<0$ and $\gamma_Z=0$, or equivalently, $\alpha_Z =-1$, that is, the origin is a non hyperbolic fixed point for the first return map $\phi_Z$. \label{itm:C}
	\end{itemize}
	
	More precisely, we will prove the following theorem:
	
	\begin{theo}[Local $\s-$structural stability on $\Omega_1$] \label{theo:B} 
		Let $Z \in \Omega_1=\Omega \setminus \Omega_0$. Then $Z \in \Xi_1$, that is, it has a condimension one singularity at the origin (see definition \ref{def:cod1}) if, and only if, $Z$ satisfies one of the following conditions:
		\begin{enumerate}[A.]
			\item $X_i \cdot Y_i(\0)<0$, $\det{Z}(\0)=0$ and $\displaystyle \frac{\partial }{\partial x_i} \det{Z}(\0) \neq 0$ for $i=1,2$;
			\item $X_i\cdot Y_i(\0)>0$, $X_j \cdot Y_j(\0)<0$, for $i,j=1,2, \, i \neq j$ and $X_1 \cdot X_2 (\0)<0$. In this case, $Z$ is transient and the coefficient $\alpha_Z$ of the first return map  \ref{eq:hopoinc} satisfies $\alpha_Z = -1$, then we ask the other coefficients to satisfy $\beta_Z \neq 0$ and $\eta_Z \neq 0$.
			\item the origin is a regular-fold to $Z$ (see \autoref{def:regfold}).
		\end{enumerate}
		In addition, the subset $\Xi^1$ is an open and dense set in $\Omega_1$, therefore local $\s-$structural stability is a generic property in $\Omega_1$.
	\end{theo}
	
	The rest of this section is devoted to prove this theorem.
	
	\begin{subsection}{The double pseudo-equilibrium bifurcation}

		We consider now the bifurcation derived from the case where $Z \in \Omega_1$ satisfies condition $A$. As $\det{Z}(\0)=0$, the origin is an pseudo-equilibrium for both sliding vector fields $Z_i^s$ (see \eqref{eq:slidingdetdef}). Moreover, as $\frac{\partial}{\partial x_j} \det{Z(\0)} \neq 0 \mbox{ for } j=1,2,$ by proposition \ref{prop:pseudoeq} we know that the origin is a hyperbolic pseudo-equilibrium for both sliding vector fields.
		
		\begin{defi} \label{def:B1}
			Let $\Xi_1^1 \subset \Omega_1$ be the subset containing all $Z$ satisfying condition $A$ of \autoref{theo:B}, equivalently, for which the origin is a hyperbolic pseudo-equilibrium of the sliding vector fields $Z_i^s, \, i=1,2$.  
			%
		\end{defi}

		
		
		\begin{prop} \label{prop:B1}
			The set $\Xi_1^1$ is an embedded codimension one submanifold of $\Omega$ and it is open in $\Omega_1$.
		\end{prop}
		
		\begin{proof}
			
			Let $Z_0 = (X^0,Y^0) \in \Xi_1^1$. Since conditions $X^0_i \cdot Y^0_i (\0)<0$ and $\frac{\partial}{\partial x_i} \det{Z_0} (\0) \neq 0$ for $i=1,2$ of \autoref{theo:B} are open, there is a neighborhood $\mathcal{V}_0 \subset \Omega$ of $Z_0$ in which these conditions are fulfilled for all $Z \in \mathcal{V}_0$. Moreover, the sign of $X_i \cdot Y_i (\0)$ and $\frac{\partial}{\partial x_i} \det{Z(\0)}$ for $i=1,2$ are constant in this neighborhood.

			Considering the Frechet differentiable  map $$\begin{array}{llll} \eta: & \mathcal{V}_0 \times \mathcal{D}_0 & \rightarrow & \R^2 \\ & (Z,(x_1,x_2)) & \mapsto &(\det{Z(0,x_1)},\det{Z(x_2,0})) \end{array}$$ where $\mathcal{D}_0$ is a neighborhood of the origin in $\R^2$ such that $X_i \cdot Y_i (p) <0$ for $p \in \s_i$ and $\sgn{\frac{\partial}{\partial x_i} \det{Z(p)}}$ is constant for $i=1,2$.

			By the Implicit Function Theorem applied to $\eta$ at the point $(Z_0,(0,0))$, we obtain a Frechet differentiable map 
			\begin{equation}\label{func:g}
				g: Z \in \mathcal{W}_0 \subset \mathcal{V}_0 \mapsto (g_1(Z),g_2(Z)) \in \U_0 \subset \mathcal{D}_0,
			\end{equation}
			satisfying $$\eta(Z,g(Z))=(\det{Z(0,g_1(Z))},\det{Z(g_2(Z),0)})=(0,0), \, \mbox{ for all } Z \in \mathcal{W}_0,$$
			where $\mathcal{W}_0$ and $\U_0$ are neighborhoods of $Z_0$ and $\0$, respectively.

			The following arguments prove simultaneously that  $\Xi_1^1$ is an embedded codimension one submanifold of $\Omega$ and an open set in $\Omega_1$.
			
			Consider the map $g_1: Z \in \mathcal{W}_0 \mapsto g_1(Z) \in \R.$   It is clear that 
			$$g^{-1}_1(0) = \Xi_1^1 \cap \mathcal{W}_0 = \Omega_1 \cap \mathcal{W}_0 \subset \Xi_1^1.$$
			In fact, given $Z \in g^{-1}_1(0)$ then $g_1(Z)=0$ and therefore $\det Z(\0)=0$, thus as $Z \in \mathcal{V}_0$,  $Z \in \Xi_1^1 \cap \mathcal{W}_0$. On the other hand, if $Z \in \mathcal{W}_0 \cap \Omega_1$ then $\det Z(\0)=0$, by the exposed above, $g_1(Z)$ is the unique point such that $\det{Z(0,x_2)}=0$, then $g_1(Z)=0$, what proves the desired equality. 
			
			To finish the proof one needs to show that $Dg_1(Z_0) \neq 0$. In fact, using the chain rule to the map $\det Z(0,g_1(Z))=0$, we obtain $$Dg_{Z_0}= \frac{D_Z \det{Z_0}(\0)}{\partial_{x_2}(\det){Z_0}(\0)}$$ which is a non zero linear functional.
			
		\end{proof}

		Now we are going to study the unfoldings of $Z \in \Xi_1^1$. More precisely, we are going to show that every local unfolding of $Z$ has exactly the same behavior.
		
		Let $Z \in \Xi_1^1$ and $\mathcal{W}_0 \subset \Omega$ be the neighborhood of $Z$ given in \ref{func:g} where $g$ is defined. One can always suppose that $\mathcal{W}_0$ is connected, therefore the  submanifold $\Xi_1^1$ splits $\mathcal{W}_0$ into two connected open subsets $ \, \mathcal{W}_0^\pm = g^{-1}_1(\R^\pm)$.
		
		Let $\gamma(\delta)= Z_\delta$ a versal unfolding of $Z$. With no loss of generality, suppose that $\gamma((-\delta_0,0)) \subset \mathcal{W}_0^-$ and $\gamma((0,\delta_0)) \subset \mathcal{W}_0^+.$
		
		It follows from \autoref{prop:B1} that for $\delta \neq 0$ the sliding vector fields $(Z_\delta)_i^s$, $i=1,2$ have a unique pseudo-equilibrium $P_i(\delta)=g_i(Z_\delta) \in \s_i, \, i=1,2$, which has the same stability as the origin has for $Z_i^s$. Moreover, if $\delta \neq 0$ then $\det{Z_\delta(\0)}\neq 0$, therefore the origin is a regular point of the sliding vector fields $(Z_\delta)_i^s$.
		
		To simplify our analysis, fix $X_i(\0)>0$ for $i=1,2$ (and therefore, $Y_i(\0)<0$ to satisfy condition $A$ of theorem \ref{theo:B}), the other case is just the reflection through the $y-$axis. This case leads to four non equivalent configurations, depending on the stability of the origin for the sliding vector fields $Z^s_i$, $i=1,2$.  However, we focus only in two, the others can be obtained analogously.
		
		Differentiating \ref{eq:slidingdef} and using  \autoref{prop:B1} it follows that \begin{equation} 
			\sgn{((Z_\delta)_i^s)'(P_i(\delta))}=(-1)^{i-1} c_i,
		\end{equation} 
		where $c_i = \displaystyle{\frac{\partial}{\partial x_j} \det{Z}(\0)}, \, i,j=1,2 \, \mbox{and } \, i \neq j .$  
		
		Moreover, by \autoref{prop:sliding}, \begin{equation} 
			\sgn{(Z_\delta)_i^s(\0)}=(-1)^{i-1} \det{Z_\delta}(\0).
		\end{equation}
		
		\begin{figure}[!htbp]
			\centering
			\begin{tiny}
				\subfigure[\label{fig:D2} Double pseudo-Equilibrium: $X_i(\0)>0$ and $c_i=1$ for $i=1,2$.]{
					\def\svgscale{0.25}
					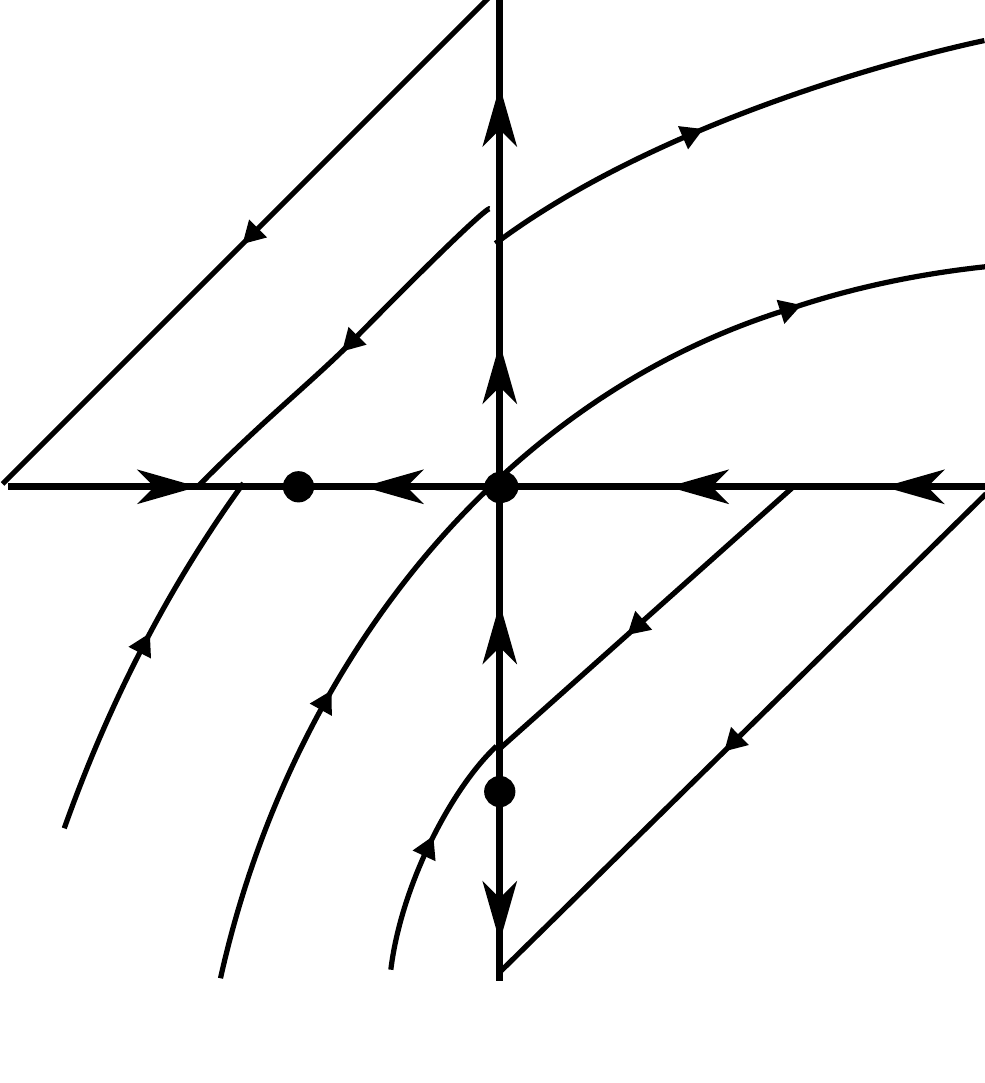
					\hspace{1cm}
					\def\svgscale{0.25}
					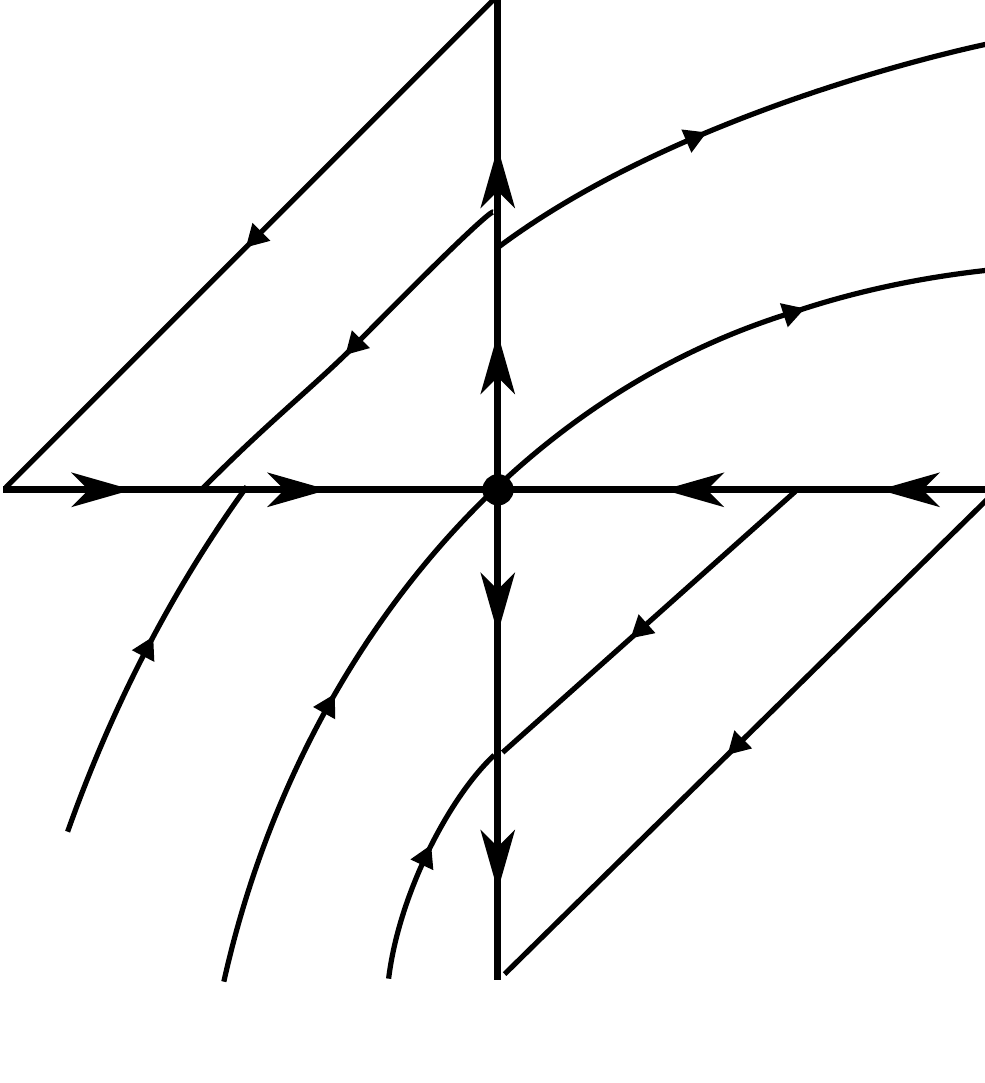
					\hspace{1cm}
					\def\svgscale{0.25}
					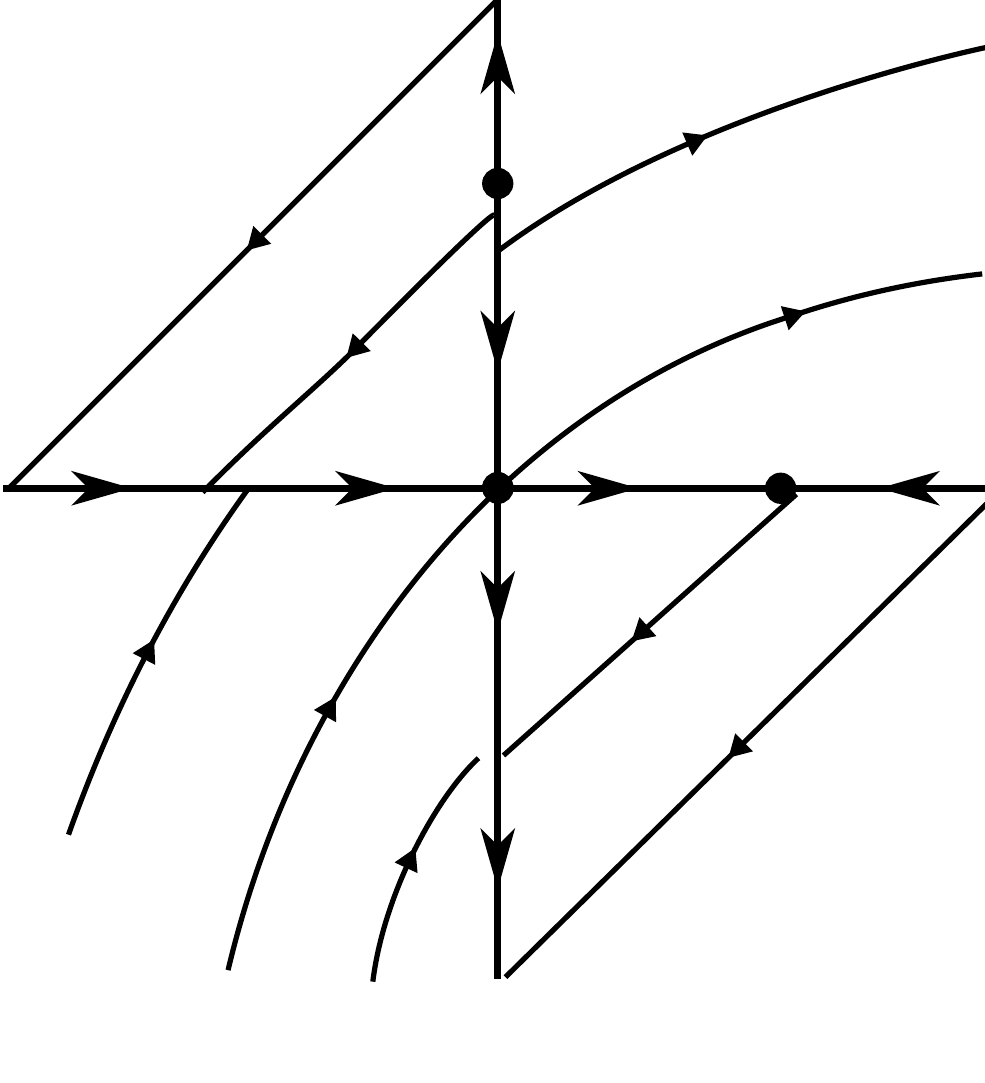}
				\subfigure[\label{fig:D1} Double pseudo-Equilibrium: $X_i(\0)>0$ for $i=1,2$ and $c_1=1$ and $c_2=-1$.]{
					\def\svgscale{0.25}
					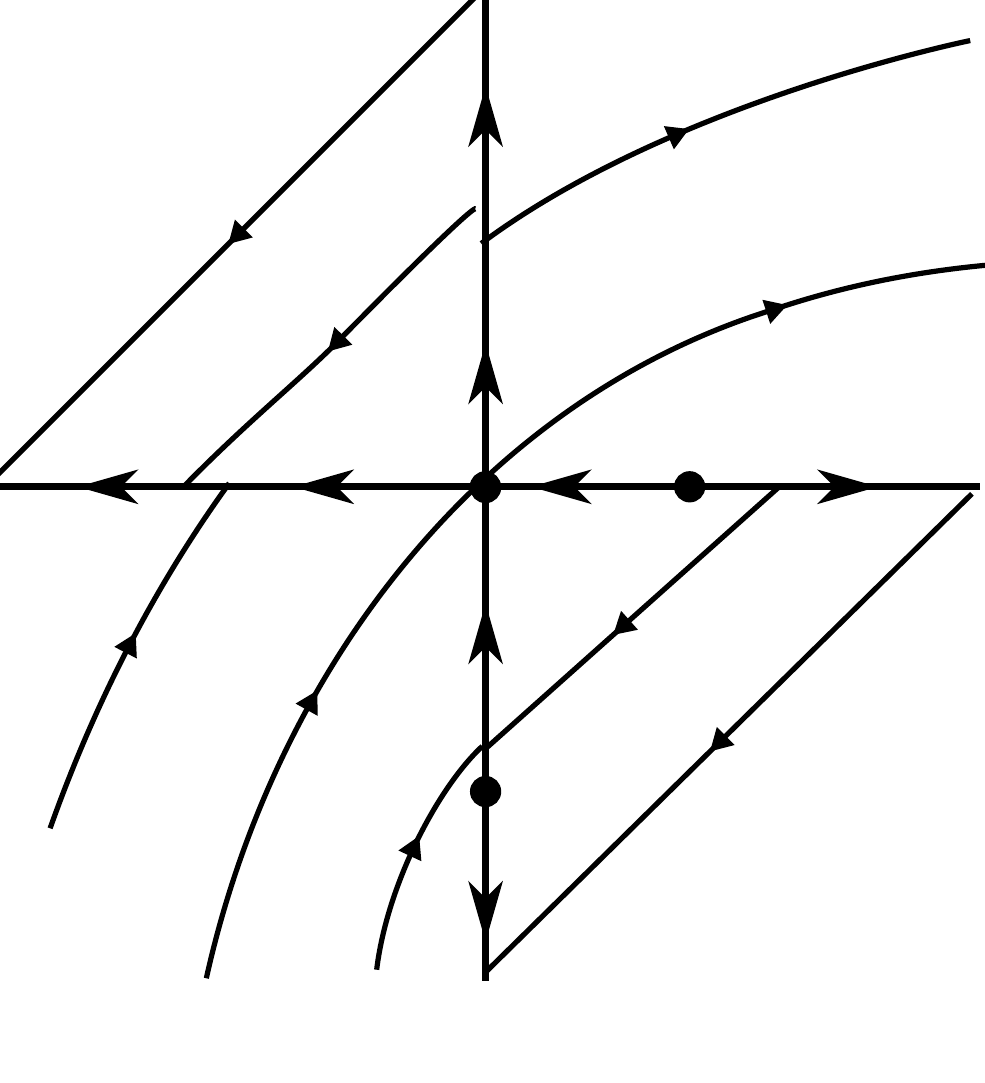
					\hspace{1cm}
					\def\svgscale{0.25}
					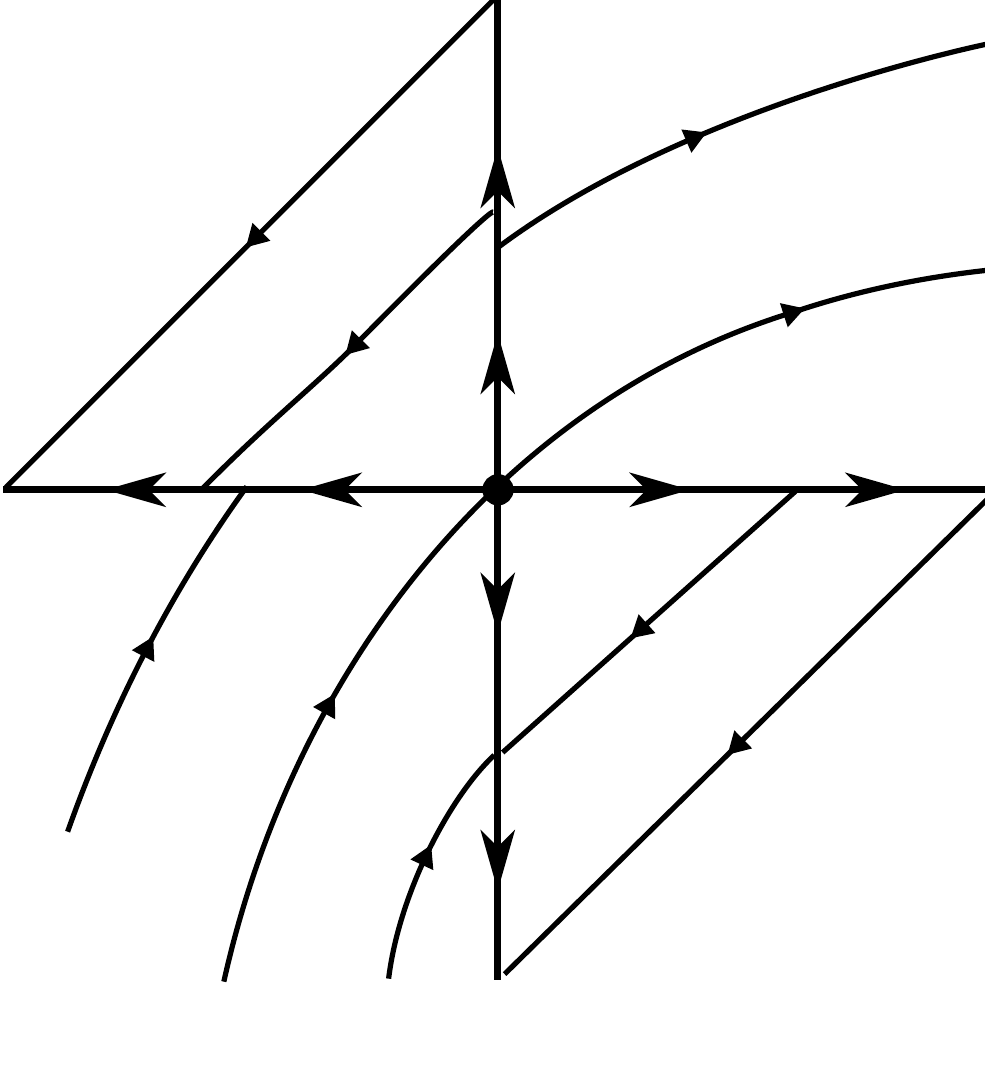
					\hspace{1cm}
					\def\svgscale{0.25}
					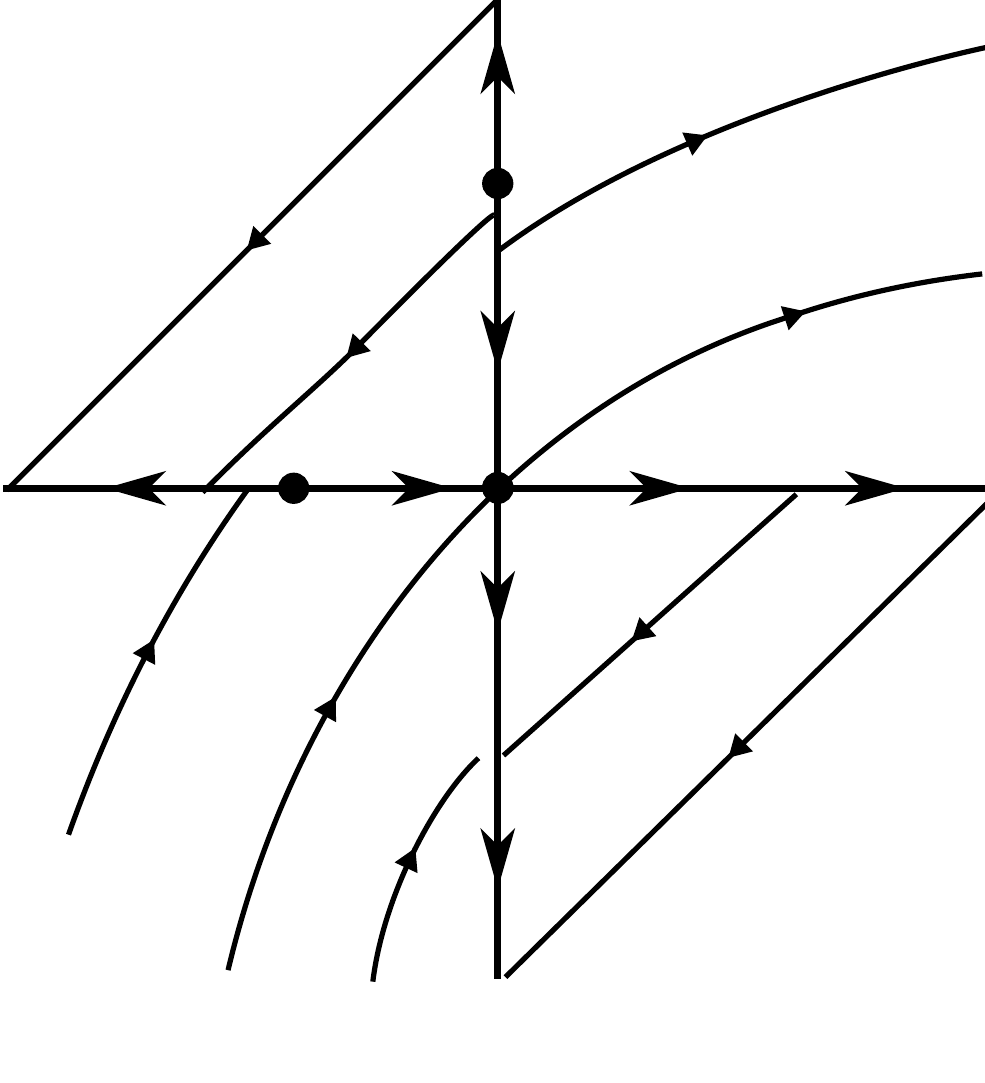}
			\end{tiny}
			\caption{The unfolding of the double-equilibrium singularity}
			\label{fig:D}
		\end{figure}
		
		One must consider four different cases depending on the signs of $c_1$ and $c_2$. We focus in only two, $c_1=c_2=1$ and $c_1=1$ and $c_2=-1$.

		$\bullet$ If $c_1=c_2=1$, for $\delta=0$ the origin is an unstable pseudo-equilibrium for $Z_1^s$ and stable for $Z_2^s$, as in \autoref{fig:D2}. In one hand, if $\delta<0$, then $Z_\delta \in \mathcal{W}_0^-$ and thus unstable pseudo-equilibrium $P_1(\delta)=g_1(Z_\delta) \in \s_1^-$. Consequently, $(Z_\delta)_1^s(0)>0$, therefore $\det Z_\delta(\0)>0$, thus $(Z_\delta)_2^s(0)<0$. Then the stable pseudo-equilibrium $P_2(\delta)$ belongs to $\s_2^-$. On the other hand, if $\delta>0$, then $Z_\delta \in \mathcal{W}_0^+$ therefore $P_1(\delta) \in \s_1^+$. It follows that $(Z_\delta)_1^s(0)<0$ and by the same argument we have $(Z_\delta)_2^s(0)>0$. Concluding that the attracting pseudo-equilibrium $P_2(\delta)$ belongs to $\s_2^+$.
		
		
		$\bullet$ In the case $c_1=1$ and $c_2=-1$, for $\delta=0$ the origin is an unstable pseudo-equilibrium for both sliding vector fields, as in \autoref{fig:D1}. Using the same argument as in the case $c_1=c_2=1$ we conclude that for $\delta<0$ then $P_1(\delta) \in \s_1^-$ and $P_2(\delta) \in \s_2^+$ and for $\delta>0$ then $P_1(\delta) \in \s_1^+$ and $P_2(\delta) \in \s_2^-.$
		
		As a consequence of the above discussion, we have the following proposition:
		
		\begin{prop} \label{C21unfolding}
			Let $Z \in \Xi_1^1$. Then any versal unfolding of $Z$ is locally weak equivalent to the one parameter family 
			\begin{equation}
				\tilde{Z}_\alpha = \begin{cases}	
					\tilde{X}_\alpha (x_1,x_2) = \left( \begin{array}{c} a- b  c_2 x_1 \\ b +a \alpha \end{array} \right), & x_1 \cdot x_2>0 \\
					\tilde{Y} (x_1,x_2) = \left( \begin{array}{c} -a \\ - b+ a  c_1 x_2 \end{array} \right), & x_1 \cdot x_2<0
				\end{cases}
			\end{equation}
			where $a=\sgn{X_1(\0)}$, $b=\sgn{X_2(\0)}$, $c_i=\sgn{\frac{\partial}{\partial x_j} \det{Z(\0)}}$ with $i,j=1,2$ and $i \neq j.$
		\end{prop}
		
		It is easy to see that the family $\tilde {Z}_\delta$ satisfies $\sgn{\tilde{X}_1 (\0)}=a,$ $\sgn{\tilde X_2(\0)}=b$, $\sgn{\frac{\partial}{\partial x_j} \det{\tilde Z(\0)}}=c_i$ with $i,j=1,2$ and $i \neq j.$ Therefore, the vector field $Z$ and $\tilde Z_0$ are locally weak equivalent. 
		
	\end{subsection}
	
	\begin{subsection}{The pseudo-Hopf bifurcation} \label{pshopf}
		
		In this section we are going to study what happens near a piecewise smooth system $Z \in \Omega_1$ satisfying condition $B$ of theorem \ref{theo:B}. 
		
		As $X_1 \cdot X_2 (\0)<0$, the vector field $Z$ is transient and we can consider the first return map $\phi_Z$, see \ref{eq:poincare}. In \autoref{prop:poinctrans} we have computed the first term of the Taylor expansion of $\phi_Z$ near the origin. Even if one can compute the higher orders terms for $\phi_Z$, they have a cumbersome expression and this computation will be omitted. 
		However, suppose that \begin{eqnarray*}
			\phi_X(x)&=&a_X x + b_X x^2 + c_X x^3 + \mathcal{O}(x^4) \\
			\phi_Y(x)&=&a_Y x + b_Y x^2 + c_Y x^3 + \mathcal{O}(x^4).
		\end{eqnarray*}
		In fact, in \autoref{prop:poinctrans} we have seen that $a_X = - \frac{X_1}{X_2}(\0)$ and $a_Y = - \frac{Y_2}{Y_1}(\0)$. Therefore, we obtain the following expression for $\phi_Z$ 
		\begin{equation} \label{eq:hopoinc}
			\phi_Z(x)=\alpha_Z^2 x + (\alpha_Z + \alpha_Z^2)\beta_Zx^2+\eta_Z x^3 + \mathcal{O}(x^4), \, x \in \s_2^-
		\end{equation}
		\noindent 
		where $\alpha_Z= \frac{X_1 Y_2}{X_2 Y_1}(\0)$, as given in  \autoref{prop:poinctrans} and we are assuming that $\alpha_Z=-1$. In this case, the coefficients $\beta_Z$ and $\eta_Z$ are given by $\beta_Z=b_X \cdot a_Y^2+a_X \cdot b_Y$ and $$\eta_Z=- 2 \left( \left( b_X \cdot a_Y \right)^2 + c_X \cdot a_Y^3 + \left( \frac{b_Y}{a_Y} \right)^2 + \left( \frac{c_Y}{a_Y} \right)^2 \right).$$
		
		It is clear that as $\alpha_Z =-1$ the origin is an unstable fixed point for $\phi_Z$ if $\eta_Z>0$ and it is stable if $\eta_Z<0.$
		
		\begin{defi} Let $\Xi_1^2$ be the set  $Z \in \Omega$ satisfying condition $B$ of \autoref{theo:B}. 
		\end{defi}
		
		\begin{prop} \label{prop:B2}
			The set $\Xi_1^2 \subset \Omega_1$ is an embedded codimension-one submanifold of $\Omega$. Consequently, it is an open set in $\Omega_1$.
		\end{prop}
		
		\begin{proof} Given $Z_0 \in \Xi_1^2$, our aim is to find a map $g: \mathcal{W}_0 \rightarrow \R$ for which $0$ is a regular value, $\mathcal{W}_0$ is a neighborhood of $Z_0$ and $ \Xi_1^2 \cap \mathcal{W}_0 = g^{-1}(0)$. 
			
			Since conditions $\beta_Z \neq 0$ and $\eta_Z \neq 0$ are open, one can find a neighborhood $\mathcal{V}_0 \subset \Omega$ of $Z_0$ such that these conditions hold, moreover, the signs of $\beta_Z$ and $\eta_Z$ are constant in $\mathcal{V}_0$.
			
			Observe that the origin is a fixed point for $\phi_Z$ for all $Z \in \mathcal{V}_0$. In this neighborhood $\phi_Z$ is written as \begin{equation*}
				\phi_Z(x) = x \left( \alpha_Z^2 + \left( \alpha_Z + \alpha_Z^2 \right) \beta_Z x + \eta_Z x^2 + \mathcal{O} \left( x^3 \right) \right)
			\end{equation*}
			
			Consider the following Frechet differentiable map $$\begin{array}{llll}
			F: 	& \mathcal{V}_0 \times  \mathcal{D}_0  & \rightarrow & \R \\
			& (Z,x) & \mapsto & \alpha_Z^2 + (\alpha_Z + \alpha_Z^2)\beta_Z x + \eta_Z x^2 + \mathcal{O}(x^3)
			\end{array}
			$$
			where $\mathcal{D}_0 \subset \s_2$ is a neighborhood of the origin.
			
			Let $G = \frac{\partial}{\partial x} F$. Since $Z_0$ belongs to $ \Xi_1^2$, we have $G(Z_0,0)=(\alpha_{Z_0} + \alpha_{Z_0}^2)\beta_{Z_0} =0$ and $\frac{\partial}{\partial x}G(Z_0,0)=2\eta_{Z_0} \neq 0$. Then by the Implicit Theorem Function there exists a Frechet differentiable map $$g: Z \in \mathcal{W}_0 \subset \mathcal{V}_0 \mapsto g(Z) \in \U_0 \subset \mathcal{D}_0$$ such that $G(Z,g(Z))=0$ for all $Z \in \mathcal{W}_0$, and $\mathcal{W}_0$ and $\U_0$ are neighborhood of $Z_0$ and $0 \in \s_2$, respectively.
			
			Then for all $Z \in \mathcal{W}_0$ we have \begin{equation} \label{eq:azexp}
				\alpha_Z + \alpha_Z^2 =\frac{-2\eta_Z}{\beta_Z}g(Z) + \mathcal{O}(g(Z)^2)
			\end{equation} 
			
			Notice that if $Z \in  \Xi_1^2 \cap \mathcal{W}_0$ then $\alpha_Z=-1$ and then $G(Z,0)=0$ thus $g(Z)=0$. On the other hand, if $g(Z)=0$, by \autoref{eq:azexp} we have that $\alpha_Z + \alpha_Z^2=0$ therefore $\alpha_Z=-1$. Thus $Z \in  \Xi_1^2$ implying that $g^{-1}(0) =  \Xi_1^2 \cap \mathcal{W}_0.$
			
			To finish the proof we need to show that $Dg_{Z_0}$ is different from zero. In fact, by the Chain rule we obtain: \begin{equation*} 
				Dg_{Z_0} =  \frac{DG_{(Z_0,0)}}{2\eta_{Z_0}} = -\frac{\beta_{Z_0}}{2\eta_{Z_0}} \cdot (D\alpha_Z)_{Z_0} \neq  0.
			\end{equation*}
		\end{proof}

		Let $Z \in  \Xi_1^2$ and consider $\gamma(\delta) = Z_\delta$ a versal unfolding of $Z$. Suppose that $\gamma$ is transverse to $ \Xi_1^2$ at $\gamma_0=Z$ and that $Z_\delta \in \mathcal{W}_0^+$ if $\delta>0$ and  $Z_\delta \in \mathcal{W}_0^-$ if $\delta<0$, where $\mathcal{W}_0^\pm = g^{-1}(\R^\pm)$.
		
		For each $\delta \in (-\delta_0,\delta_0)$ we associate the first return map $$\phi_\delta(x)=\alpha_\delta^2 x +(\alpha_\delta+\alpha_\delta^2)\beta_\delta x^2 +\eta_\delta x^3 +\mathcal{O}(x)^4.$$ 
		

		Observe that $x \neq 0$ is a fixed point for $\phi_\delta$ if, and only if, $F(Z_\delta,x)=1$. The point $g(Z_\delta) \in \s_2^-$ given by \autoref{prop:B2} is a critical point of $F_\delta(x)=F(Z_\delta,x)-1$ since $$(F_\delta)'(g(Z_\delta))=\frac{\partial}{\partial x} F(Z_\delta,g(Z_\delta))=G(Z_\delta,g(Z_\delta))=0.$$ Therefore $g(Z_\delta)$ is a local maximum or minimum of $F_\delta$ depending on $\sgn{ \eta_\delta}= \sgn{ \eta_Z}$ if $\delta$ is small enough.
		
		\begin{figure}[!htbp]
			\centering
			\begin{tiny}
				\def\svgscale{0.25}
				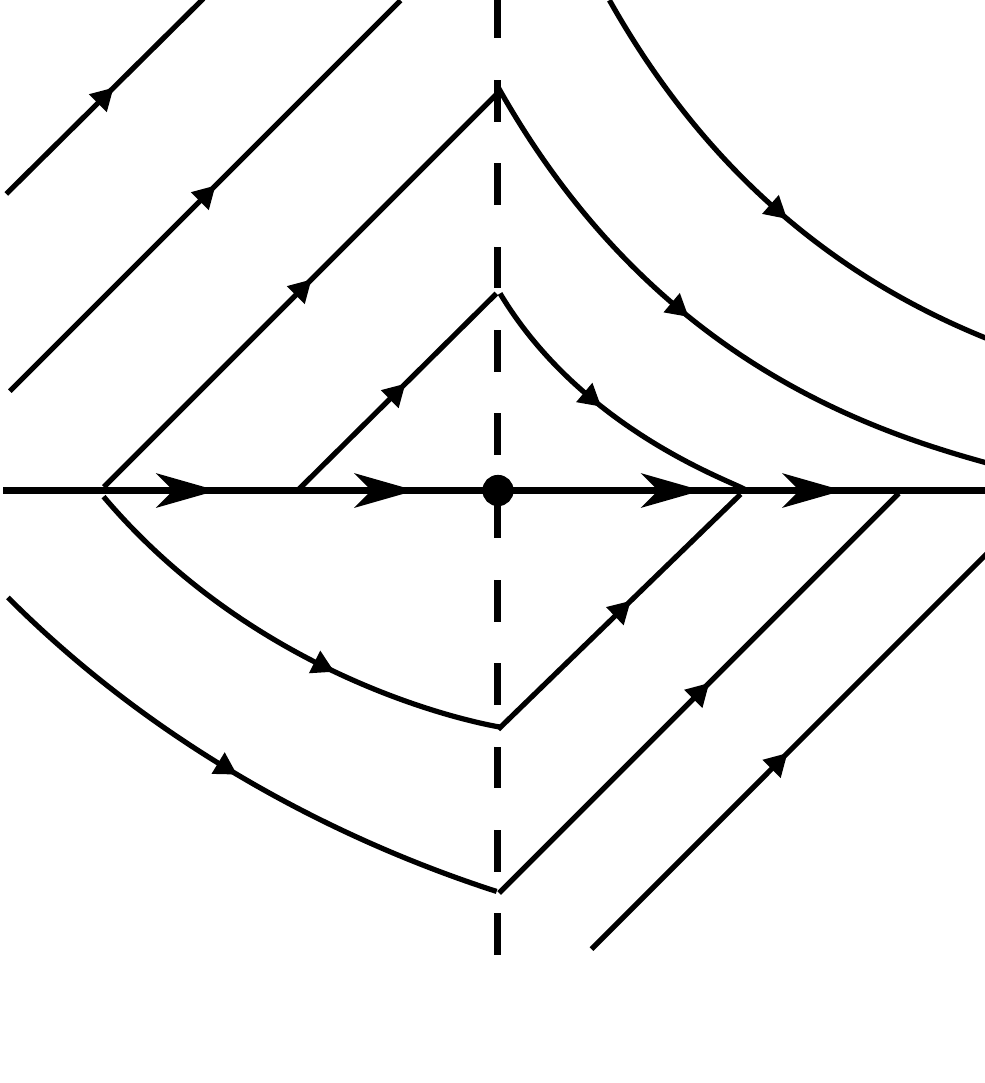
				\hspace{1cm}
				\def\svgscale{0.25}
				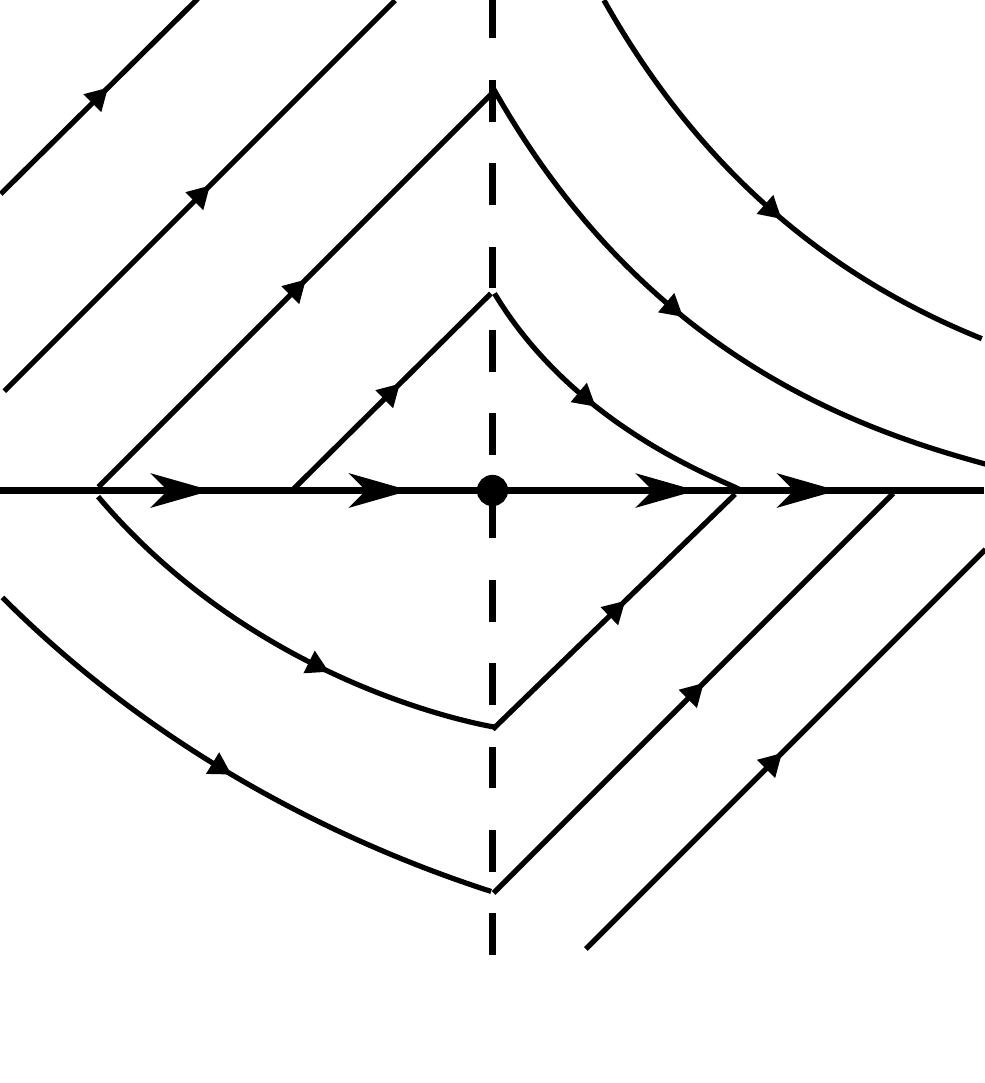
				\hspace{1cm}
				\def\svgscale{0.25}
				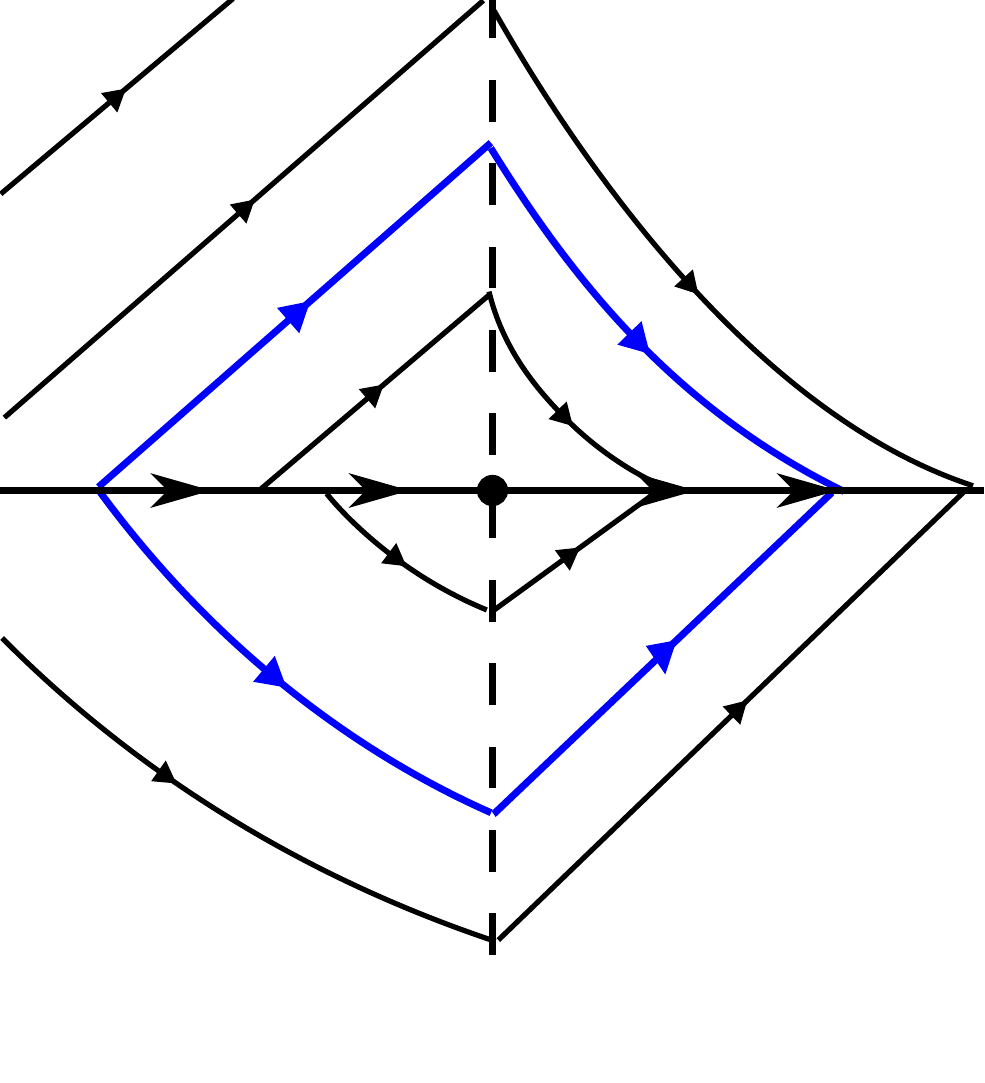
			\end{tiny}
			\caption{The pseudo-Hopf bifurcation for $Z$ satisfying $\eta_Z>0$.}
			\label{fig:PH}
		\end{figure}
		
		To simplify the analysis, suppose that $X_1(\0)>0$. Fix $\eta_Z$ and $\beta_Z>0$. Therefore, the origin is a repelling fixed point of $\phi_Z$ and the point $g(Z_\delta)$ is a minimum of $F_\delta(x)$.
		
		\begin{rem} The case $\beta_Z<0$ is just a reparametrization of $\gamma$ making $\delta \mapsto -\delta$.
		\end{rem}

		If $\delta>0$, by \autoref{eq:azexp} we have $\alpha_\delta + \alpha_\delta^2 <0$, thus $-1<\alpha_\delta<0$ therefore the origin is an stable fixed point of $\phi_\delta$.

		Moreover using \autoref{eq:azexp}, one can see that  $F_\delta(g(Z_\delta))=\alpha_\delta^2- \eta_\delta g(Z_\delta)^2 + \mathcal{O}(g(Z_\delta)^3)<0$ and therefore there exist unique points $p^-_\delta<0<p^+_\delta$ satisfying $F_\delta(p_\delta^\pm) =0$ which are unstable fixed points of $\phi_\delta$ since $(\phi_\delta)'(p_\delta^\pm)>1.$ Associated to this fixed point there is a pseudo-cycle of $Z_\delta$.

		When $\delta<0$ we obtain  $\alpha_\delta + \alpha_\delta^2 > 0$ then the origin is an unstable fixed point of $\phi_\delta$ since $\alpha_\delta < -1$. In this case we have $F_\delta(g(Z_\delta))>0$ and therefore there are no fixed points of $\phi_\delta$ around the origin.
		
		The same reasoning can be done when $\eta_Z<0$, in this case $g(Z_\delta)$ is a maximum. It follows that for $\delta>0$ the origin becomes unstable. Moreover, $F(g(Z_\delta))>0$ and being a maximum, two stable fixed appears, given rise to a pseudo-cycle of $Z_\delta$. For $\delta<0$ the origin remains stable and there are no pseudo-cycle around the origin.
		
		\begin{prop} \label{C31unfolding} 
			Let $Z \in \Xi_2^1$. Then any versal unfolding of $Z$ is locally weak equivalent to the one parameter family 
			\begin{equation}
				\tilde{Z}_\delta = \begin{cases}	
					\tilde{X}_\delta (x_1, x_2) =  \left( \begin{array}{c} ac \\ -(ac + \delta) \end{array} \right), & x_1 \cdot x_2>0 \\
					\tilde{Y} (x_1,x_2) = \left( \begin{array}{c} a \\ a+ x + ab x_1^2 \end{array} \right), & x_1 \cdot x_2<0
				\end{cases}
			\end{equation}
			where $a=\sgn{Y_1(\0)}$, $b=\sgn{\eta_Z}$, $c=\sgn{X_1 \cdot Y_1 (\0)}$.
		\end{prop}
		
		\begin{proof}
			It is enough to observe that $\alpha_{Z_{\delta}} =-1$, $\beta_{Z_{\delta}} \neq 0$ and $\sgn{\eta_{Z_{\delta}}} =\sgn{\eta_{Z}}.$
		\end{proof}
		
	\end{subsection}
	
	\begin{subsection}{Regular-fold bifurcation}
		
		In this section we are going to study what happens near a piecewise smooth system $Z \in \Omega_1$ satisfying condition $C$ of theorem \ref{theo:B}.
		
		In this case, we allow the trajectories of $X$ or $Y$ being tangent to $\s$ at the origin. If  the trajectory through the origin of $X$, is tangent to both $\s_i$ for $i=1,2$ then $X$ has a singularity at the origin and this situation leads to a higher codimension bifurcation. Another situation which raises a higher codimension bifurcation is when the trajectories through zero of both vector fields $X$ and $Y$ are tangent to $\s.$
		
		In this section we will see that $Z$ having a regular fold is a codimension one bifurcation. 
		
		\begin{defi} \label{def:regfold1}
			We call $\Xi_1^3$ the set of all elements of $\Omega_1$  having a regular fold at the origin.
		\end{defi} 
		
		Without loss of generality, we can fix $X_2(\0)=0$, $X_1(\0)>0$ and $\frac{\partial}{\partial x_1}X_2(\0)>0$. Under these conditions we have four cases to analyze, which are not equivalent, depending on sign of $Y_1(\0)$ and $Y_2 (\0)$. The other ones can be obtained by reversing time and combinations of reflections and rotations of the previous ones. 
		
		\begin{rem} When $Z \in \Xi_1^3$, provided $Z^s_i$ is defined in $\overline{\s_i}$ for $i=1,2$, we have $\displaystyle Z_1^s(0)=\frac{(X_1\cdot Y_2)(\0)}{(X_1-Y_1)(\0)}$ and $Z_2^s(0)=X_1(\0)$, therefore, the origin is a regular point for the sliding vector fields $Z^s_i$.
		\end{rem}
		

		\begin{prop} \label{prop:B3}
			The set $\Xi_1^3$ is an embedded codimension-one submanifold of $\Omega$ and an open set of $\Omega_1.$ 
		\end{prop}
		
		\begin{proof}
			Let $Z_0=(X^0,Y^0) \in \Xi_1^3.$ In order to simplify the notation, we will make the proof for the case when the origin is a fold of $X^0$ at $\s_2$. Any other case is analogous.
			
			Let $\mathcal{V}_0 \times \mathcal{D}_0 \subset \Omega \times \U$ a connected neighborhood of $(Z_0, \0)$ for which all $Z \in \mathcal{V}_0$ satisfies:
			\begin{enumerate}[(a)]
				\item $\sgn{\det Z(p)} = \sgn{\det{Z_0(\0)}}$ for all $p \in \mathcal{D}_0 \cap \s$;
				\item $\sgn{\gamma_Z} = \sgn{\gamma_{Z_0}}$, where $\gamma_Z$ is defined in \ref{form:gamma};
				\item The sign of $X_1(x,0)$, $Y_i(x,0)$, $i=1,2$, and $\frac{\partial} {\partial x} X_2(x,0)$  is constant in $\mathcal{V}_0 \times \mathcal{D}_0$.
			\end{enumerate}
			
			Consider the Frechet differentiable application $$\begin{array}{llll} 
			\xi:& \mathcal{V}_0 \times (\mathcal{D}_0 \cap \s_2) &\rightarrow & \R \\ &(Z,x) &\mapsto & X_2(x,0)\end{array}$$
			
			We have $\xi(Z_0,\0)=0$, since the origin is a fold point to $X$ at $\s_2$ and $\frac{\partial}{\partial x} \xi (Z_0,\0) = \frac{\partial}{\partial x_1} X^0_2 (0,0) \neq 0$. Then by the Implicit Function Theorem we obtain open sets $\mathcal{W}_0 \subset \mathcal{V}_0$, $\U_0 \subset \mathcal{D}_0 \cap \s_2$ and a Frechet differentiable map $g: \mathcal{W}_0 \rightarrow \U_0$ such that $\xi(Z,x)=0$ if, and only if, $x=g(Z)$. The point $(g(Z),0)$ is the unique tangency point of $Z$ with $\s$ and it is a regular-fold of $X$ at $\s_2$.
			
			Therefore, $g^{-1}(0)=\mathcal{W}_0 \cap \Xi_1^3$. Using the Chain rule it is easy to see that $Dg_{Z_0}$ is a surjective linear functional. Therefore $\Xi_1^3$ is a codimension one embedded submanifold of $\Omega$ and also an open set of $\Omega_1$.
			
		\end{proof}
		
		\bigskip
		
		From now on we are going to present the generic unfoldings of $Z \in  \Xi_1^3$. Fixing $X_2(\0)=0$ and $X_1 \cdot \frac{\partial}{\partial x}X_2(\0)>0$, the other cases can be done identically.
		
		Let $\gamma(\delta)= Z_\delta$ a versal unfolding of $Z$. Let $\mathcal{W}_0^\pm = g^{-1}(\R^\pm)$ and without loss of generality, suppose that for $\delta>0$, $Z_{\delta} \in \mathcal{W}_0^+$, therefore, $X$ has a fold in $\s_2^+$ which is visible. Analogously, for if $\delta<0$, $Z_\delta \in \mathcal{W}_0^-$, therefore, $X$ has a fold in $\s_2^-$ which is invisible. Knowing the position of the fold for every $\delta$, it allows us to give the decomposition of $\s$. See \autoref{fig:F} and also \autoref{tab:decomposition}. \newline

		\begin{table}[htb!]
			\centering	
			\begin{tabular}{cc}
				\begin{tabular}{|c|c|c|}
					\hline
					\multicolumn{3}{|c|}{$Y_1 \cdot Y_2 (\0)>0$}  \\ \hline 
					\hline
					&$Y_1(\0)>0$ & $Y_1(\0)<0$ \\ \hline
					\begin{tabular}{c}
						$\delta>0$   \\ 
						$g(Z_\delta)>0$
					\end{tabular}
					& 
					\begin{tabular}{l}
						$\s_1=\s^c$ \\ $\s_2^-=\s^e$ \\ $\s_2^+=\s^c \cup \s^s$
					\end{tabular}		 
					& 
					\begin{tabular}{l} $\s_1^-=\s^s$ \\ $\s_1^+=\s^e$ \\ $\s_2^-=\s^c$ \\ $\s_2^+=\s^c \cup \s^e$
					\end{tabular}  
					\\ \hline
					\begin{tabular}{c}
						$\delta=0$   \\ 
						$g(Z_\delta)=0$
					\end{tabular}
					& 
					\begin{tabular}{l}
						$\s_1=\s^c$ \\ $\s_2^-=\s^e$ \\ $\s_2^+= \s^c$
					\end{tabular}		 
					& 
					\begin{tabular}{l} $\s_1^-=\s^s$ \\ $\s_1^+=\s^e$ \\ $\s_2^-=\s^c$ \\ $\s_2^+= \s^e$
					\end{tabular}  
					\\ \hline
					\begin{tabular}{l} $\delta<0$ \\ $g(Z_\delta)<0$ \end{tabular}  & 
					\begin{tabular}{l} $\s_1=\s^c$ \\ $\s_2^-=\s^e \cup \s^c$ \\ $\s_2^+=\s^c$ \end{tabular} & 
					\begin{tabular}{l} $\s_1^-=\s^s$ \\ $\s_1^+=\s^e$ \\ $\s_2^-=\s^c \cup \s^s$ \\ $\s_2^+=\s^e$ \end{tabular}
					\\ \hline
				\end{tabular}	
				\bigskip
				& 
				\begin{tabular}{|c|c|c|}
					\hline
					\multicolumn{3}{|c|}{$Y_1 \cdot Y_2 (\0)<0$}  \\ \hline 
					\hline
					&$Y_1(\0)>0$ & $Y_1(\0)<0$ \\ \hline
					\begin{tabular}{c}
						$\delta>0$   \\ 
						$g(Z_\delta)>0$
					\end{tabular}
					& 
					\begin{tabular}{l}
						$\s_1=\s^c$ \\ $\s_2^-=\s^c$ \\ $\s_2^+=\s^c \cup \s^e$
					\end{tabular}		 
					& 
					\begin{tabular}{l} $\s_1^-=\s^s$ \\ $\s_1^+=\s^e$ \\ $\s_2^-=\s^e$ \\ $\s_2^+=\s^s \cup \s^c$
					\end{tabular}  
					\\ \hline
					\begin{tabular}{c}
						$\delta=0$   \\ 
						$g(Z_\delta)=0$
					\end{tabular}
					& 
					\begin{tabular}{l}
						$\s_1^=\s^c$ \\ $\s_2^-=\s^c$ \\ $\s_2^+= \s^e$
					\end{tabular}		 
					& 
					\begin{tabular}{l} $\s_1^-=\s^s$ \\ $\s_1^+=\s^e$ \\ $\s_2^-=\s^e$ \\ $\s_2^+= \s^c$
					\end{tabular}  
					\\ \hline
					
					\begin{tabular}{l} $\delta<0$ \\ $g(Z_\delta)<0$ \end{tabular}  & 
					\begin{tabular}{l} $\s_1=\s^c$ \\ $\s_2^-=\s^c \cup \s^s$ \\ $\s_2^+=\s^e$ \end{tabular} & 
					\begin{tabular}{l} $\s_1^-=\s^s$ \\ $\s_1^+=\s^e$ \\ $\s_2^-=\s^e \cup \s^c$ \\ $\s_2^+=\s^c$ \end{tabular}
					\\ \hline
				\end{tabular}
			\end{tabular}
			\caption{Decomposition of $\s$ when $X_2(\0)=0$, $X_1(\0)\cdot \frac{\partial}{\partial x} X_2(\0)>0$, depending on signs of $Y_1(\0)$ and $Y_2(\0)$.}
			\label{tab:decomposition}
		\end{table}
		
		Since $\sgn{\det{Z_\delta(p)}}$ does not change for all $(Z,p) \in  \U \times (\U_0 \cap \s)$, whenever it is defined, $\sgn{(Z_\delta)_i^s(p)}=\sgn{Z_i^s(0)}>0$ for all $\delta \in (-\delta_0,\delta_0)$. \newline

		\begin{itemize}
			\item Suppose that $Y_1 \cdot Y_2 (\0)>0$, thus the $Y$ trajectory trough zero has positive slope.
			\begin{itemize}
				\item Let $Y_1(\0)>0$. \autoref{fig:F2} 
				\begin{itemize}
					\item 
					If $\delta<0$, $X_2 >0$ near the origin and therefore $Z_\delta$ satisfies hypothesis $A$ of theorem \ref{theo:A}. 
					\item When $\delta>0$, $X_2<0$ near the origin. In this case, the vector field $Z_\delta$ is transient  and therefore, it is necessary to analyze the stability of the origin for the first return map $\phi_\delta$. By \autoref{rem:hyppoinc}, the origin is a hyperbolic fixed point of $\phi_\delta$ if $$\gamma_\delta=\left( (X_\delta)_1 \cdot (Y_\delta)_2 + (X_\delta)_2 \cdot (Y_\delta)_1 \right) (\0) \neq 0,$$ 
					but $\gamma_\delta$ is positive for $\delta$ sufficiently small  because  for $\delta=0$ we have $\gamma_0=X_1\cdot Y_2(\0)>0$.
					Therefore, $Z_\delta$ satisfies the hypothesis $C$ of theorem \ref{theo:A} and therefore, it is locally $\s$-structurally stable. 
				\end{itemize}
				\item Let $Y_1(\0)<0,$ see \autoref{fig:F4}.
				\begin{itemize} 
					\item 
					If $\delta<0$, $X_2<0$ near the origin and therefore the origin satisfies condition $B$ of theorem \ref{theo:A}. 
					\item 
					The case $\delta>0$ is analogous to the case $Y_1(\0).$
				\end{itemize}
			\end{itemize}
			\item Suppose that $Y_1 \cdot Y_2 (\0)<0$, that is the $Y$ trajectory of zero has negative slope. We have different dynamics on the switching manifold depending on $\sgn{Y_1(\0)}.$ 
			\begin{itemize}
				\item Fix $Y_1 (\0)>0$, see figure \ref{fig:F1}.
				\begin{itemize}
					\item For $\delta<0$, $X_2>0$ near the origin and therefore, $Z_\delta$ satisfies the condition $C$ with $X_1 \cdot X_2 >0$, therefore, $Z_\delta$ it is also locally $\s$-structurally stable.
					\item If $\delta>0$, then $X_2<0$ near the origin and therefore, $Z_\delta$ satisfies hypothesis $A$ of theorem \ref{theo:A}. 
				\end{itemize}
				\item Fix $Y_1 (\0)<0$, see figure \ref{fig:F3}. 
				\begin{itemize}
					\item 
					If $\delta<0$, then $(X_\delta)_2>0$ near the origin and therefore $Z_\delta$ satisfies hypothesis $C$ of theorem \ref{theo:A} with $(X_\delta)_1 \cdot (X_\delta)_2 (\0)>0$, therefore $Z_\delta$ is locally $\s-$structurally stable. 
					\item
					If $\delta>0$ then $(X_\delta)_2<0$ near the origin and therefore, as $\det{Z_\delta}(\0) \neq 0$, $Z_\delta$ satisfies hypothesis $B$ of theorem \ref{theo:A}, thus $Z_\delta$ is locally $\s-$structurally stable. 
				\end{itemize} 
			\end{itemize}
		\end{itemize}
		
		The following proposition gives a normal form which satisfies all the previous conditions.
		
		\begin{prop} \label{C41unfolding}
			Let $Z \in \Xi_3^1$ and fix $X_1(\0)$ and $\frac{\partial}{\partial x} X_2 (\0)$ positive. Then any versal unfolding of $Z$ is locally weak equivalent to the one parameter family 
			\begin{equation}
				\tilde{Z}_\alpha = \begin{cases}	
					\tilde{X}_\alpha (x_1,x_2) = \left( \begin{array}{c} 1 \\ x_1 - \alpha \end{array} \right), & x_1 \cdot x_2>0 \\
					\tilde{Y} (x_1,x_2) = \left( \begin{array}{c} a \\ b \end{array} \right), & x_1 \cdot x_2<0
				\end{cases}
			\end{equation}
			where $a=\sgn{Y_1(\0)}$, $b=Y_2(\0)$. The other normal forms can be obtained from this proposition by reflections and rotations.
		\end{prop}
		
		Joining all the results of the last three subsections we prove \autoref{theo:B}.

				\begin{figure}[t]
				\centering
				\begin{tiny}
				\subfigure[\label{fig:F2} $Y_1 \cdot Y_2 (\0)>0$ and $Y_1(\0)>0$]{
					\def\svgscale{0.25}
					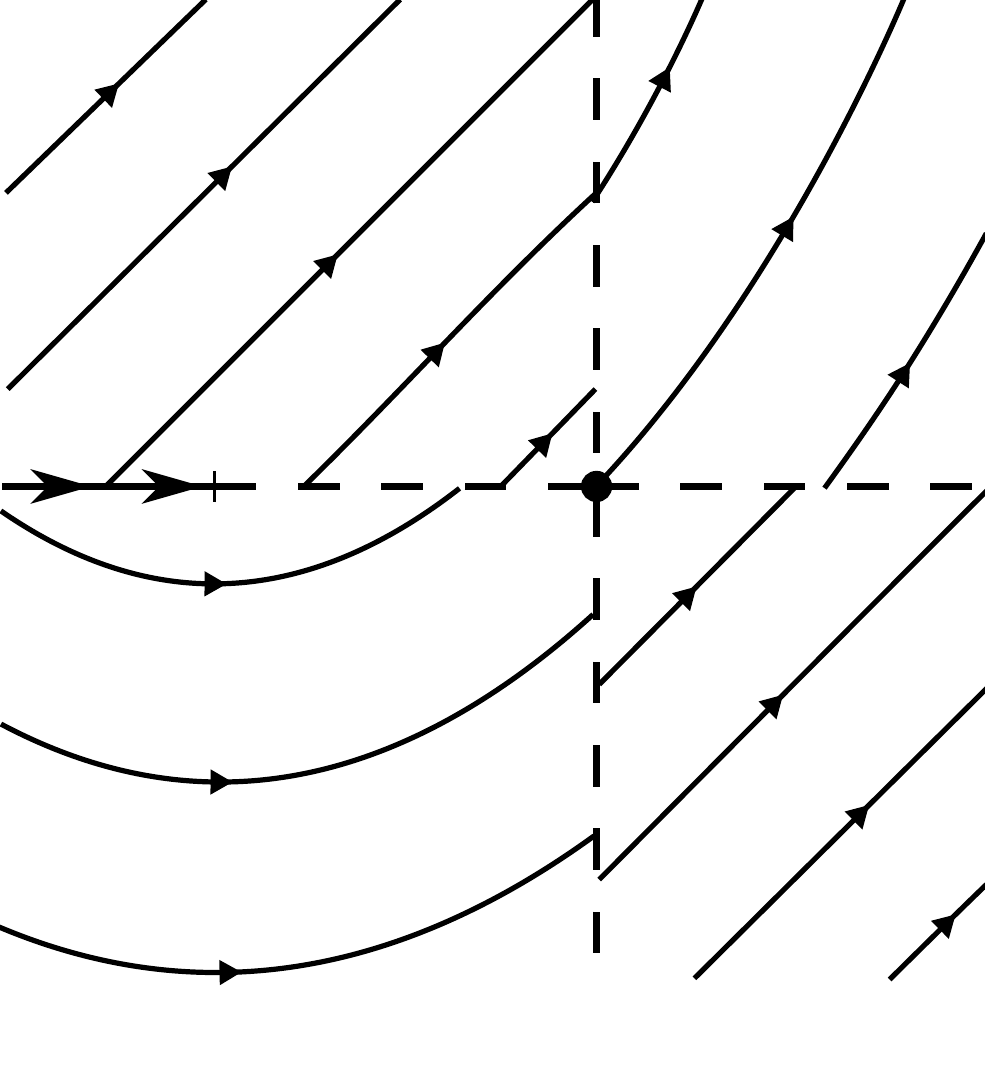 \hspace{0.5cm}
					\def\svgscale{0.25}
					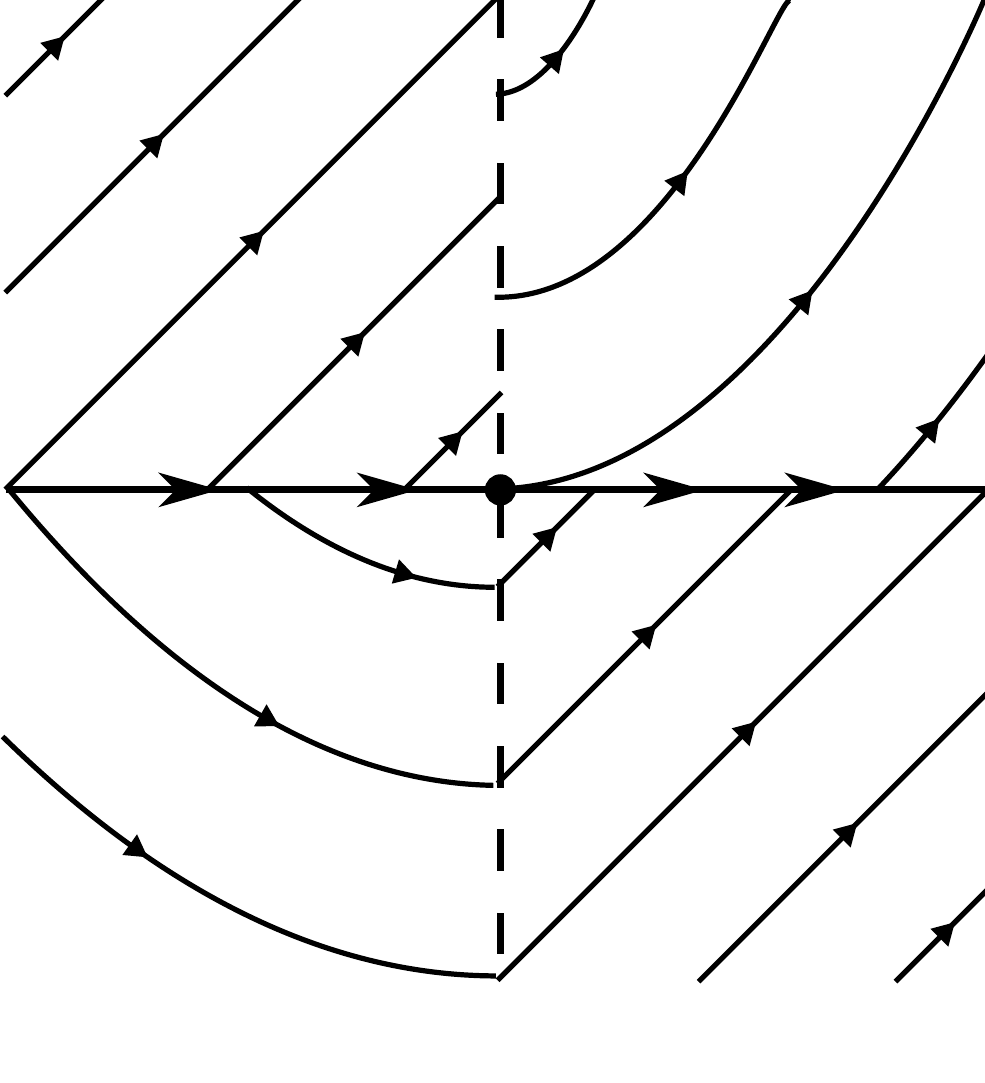 \hspace{0.5cm}
					\def\svgscale{0.25}
					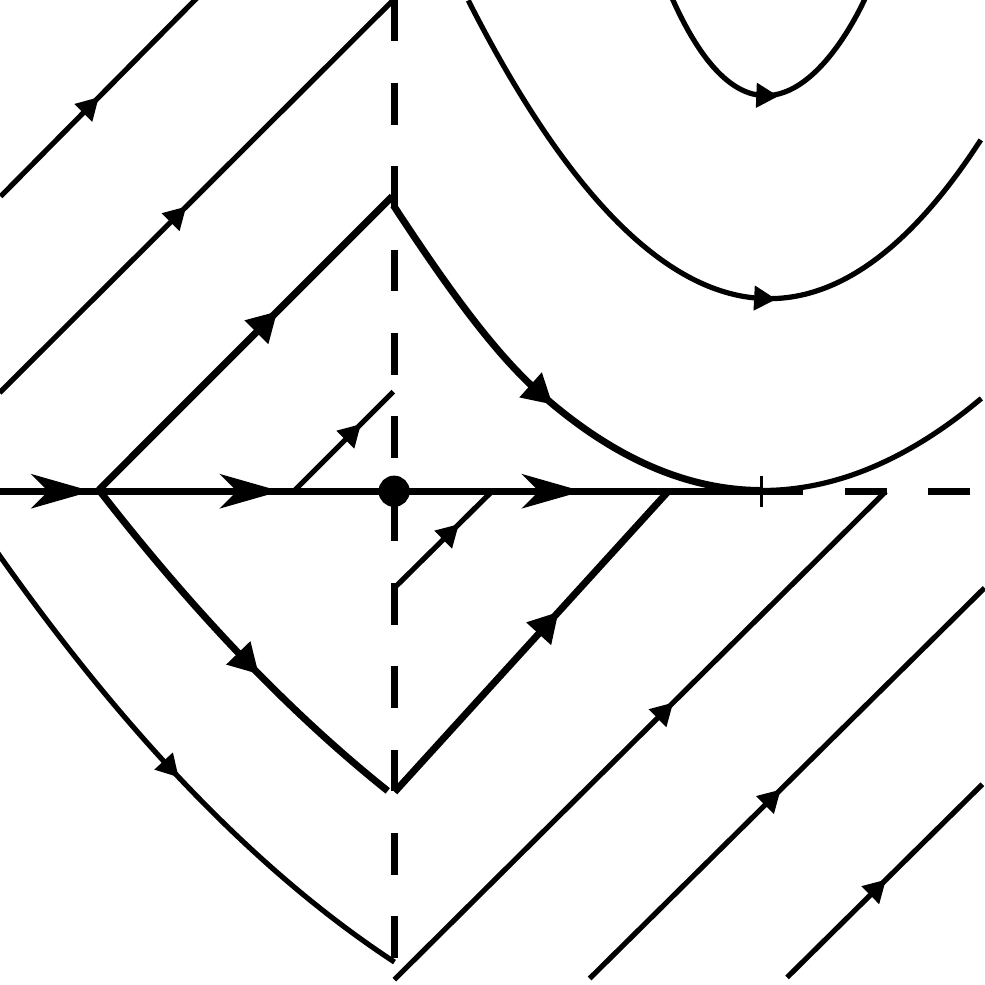}
				\vspace{0.2cm}
				\subfigure[\label{fig:F4} $Y_1 \cdot Y_2 (\0)>0$ and $Y_1(\0)<0$ ]{
					\def\svgscale{0.25}
					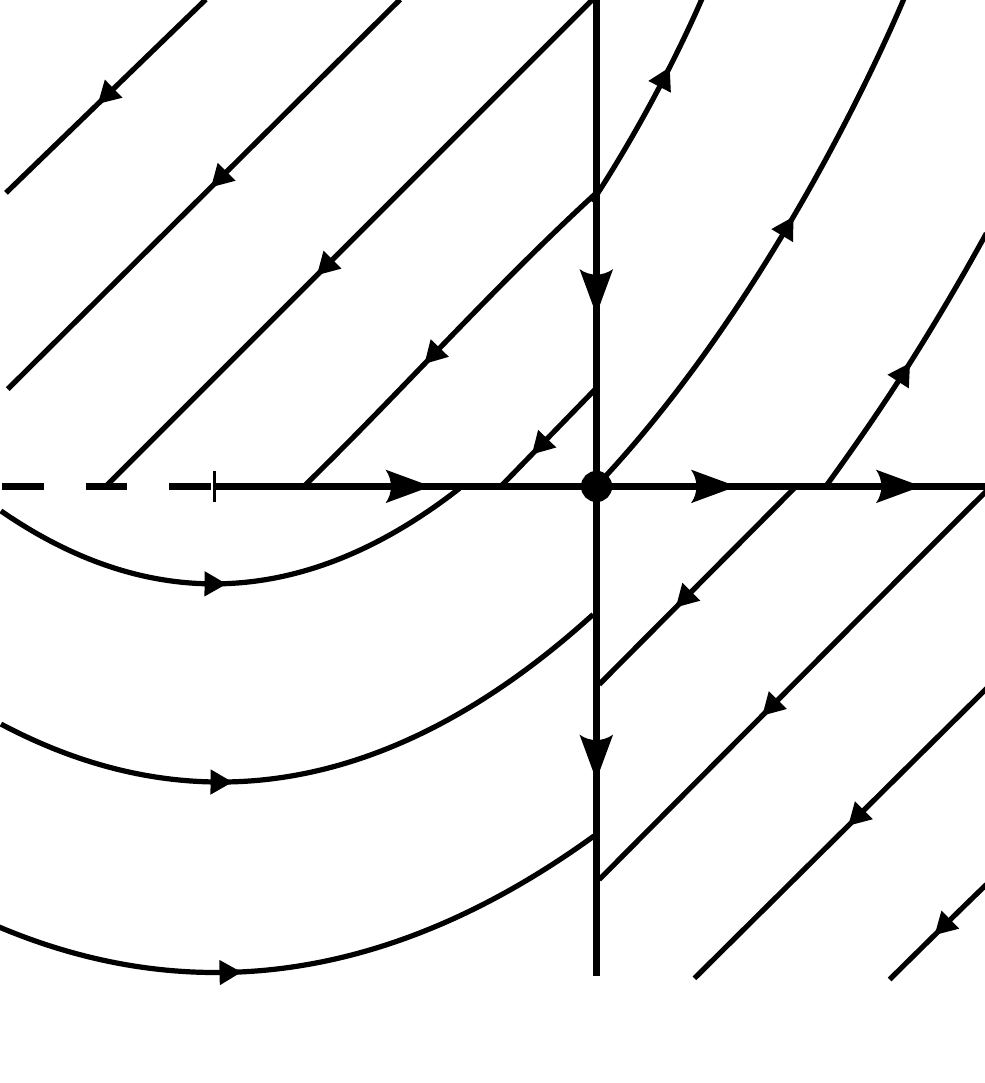 \hspace{0.5cm}
					\def\svgscale{0.25}
					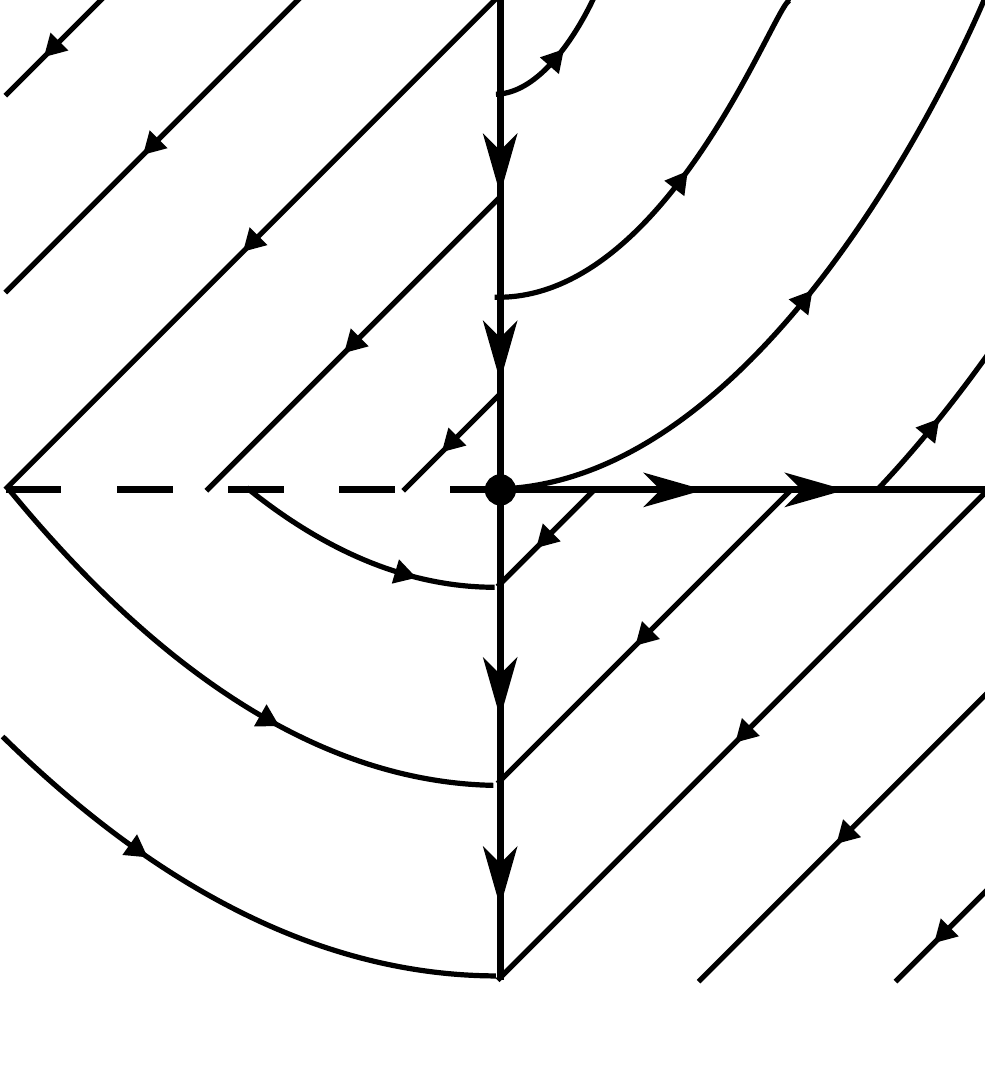 \hspace{0.5cm}
					\def\svgscale{0.25}
					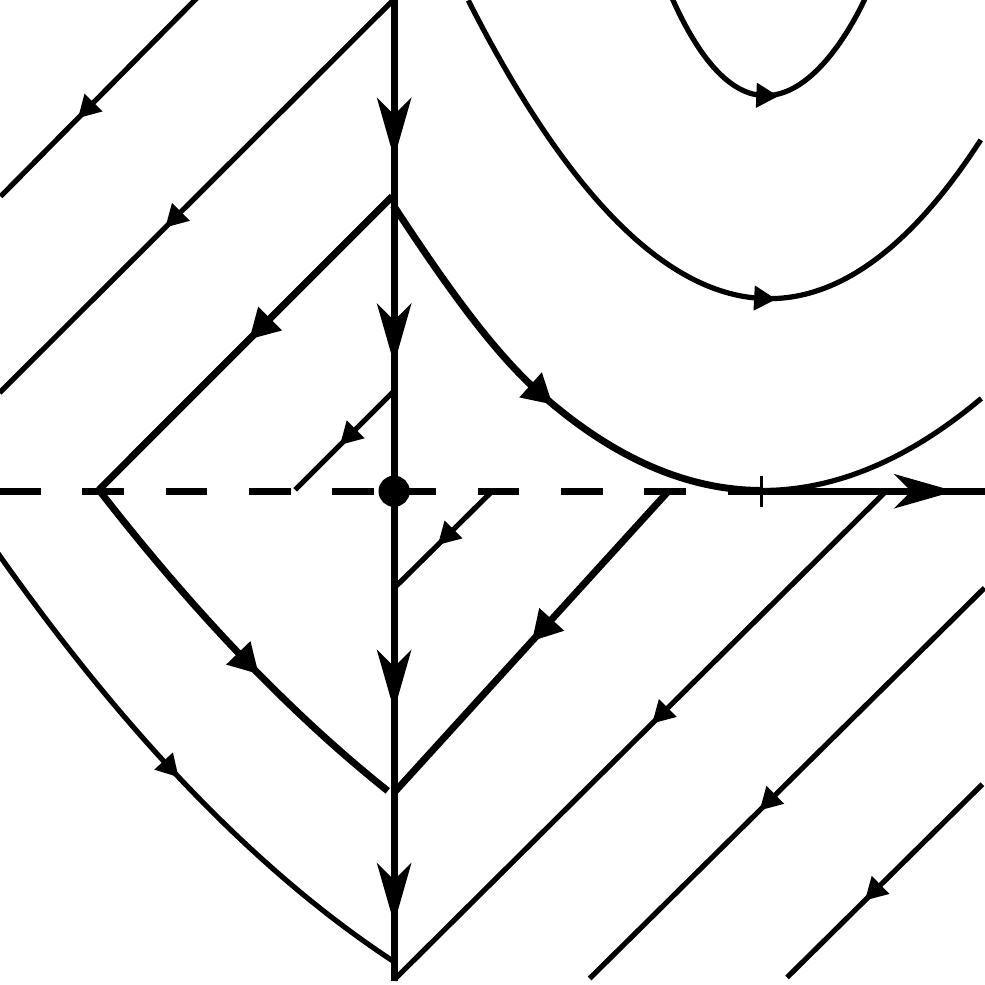}
			\end{tiny}
			\caption{The unfoldings for a regular-fold singularity satisfying $Y_1 \cdot Y_2 (\0)>0$.}
			\label{fig:F-}
		\end{figure}
	
		\begin{figure}[t]
		\centering
		\begin{tiny}
			\subfigure[\label{fig:F1} $Y_1 \cdot Y_2 (\0)<0$ and $Y_1(\0)>0$]
			{\def\svgscale{0.25}
				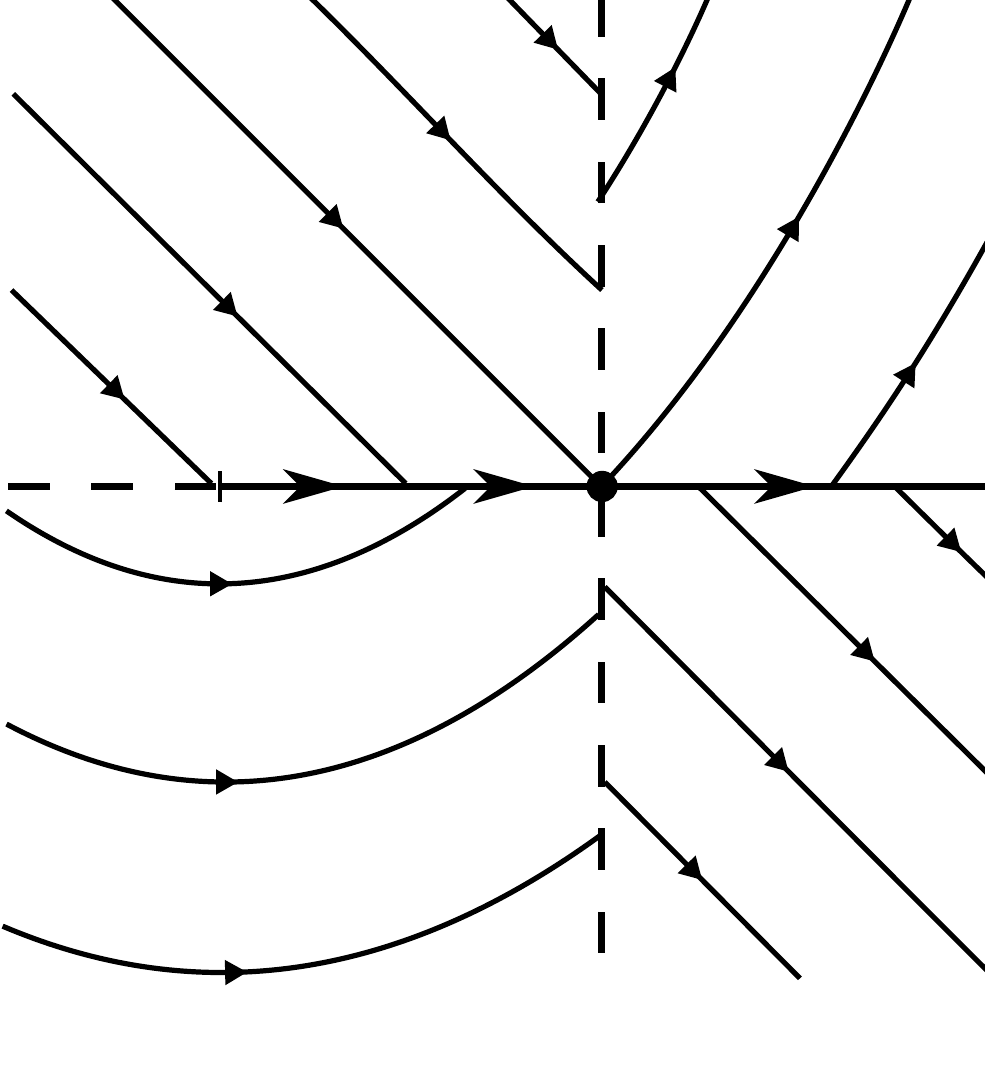 \hspace{0.5cm}
				\def\svgscale{0.25}
				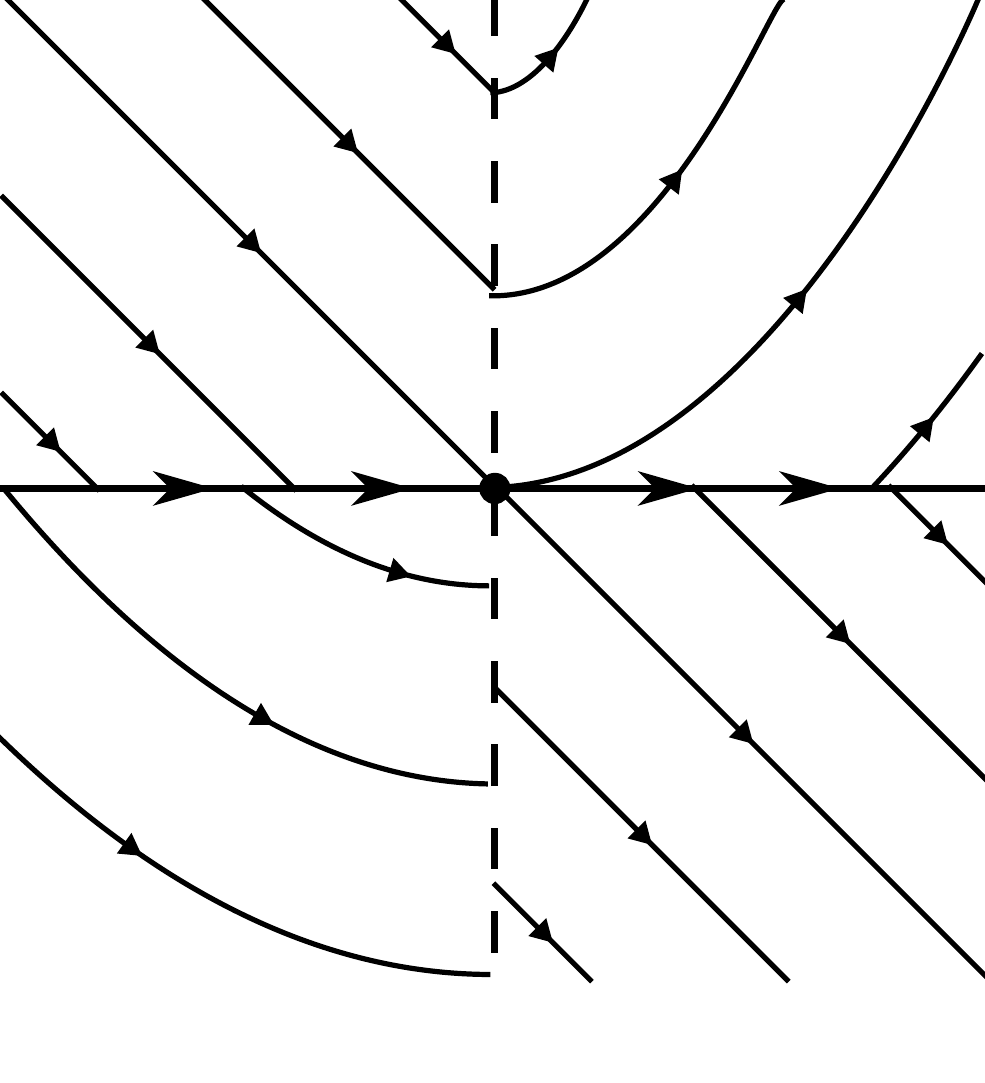 \hspace{0.5cm}
				\def\svgscale{0.25}
				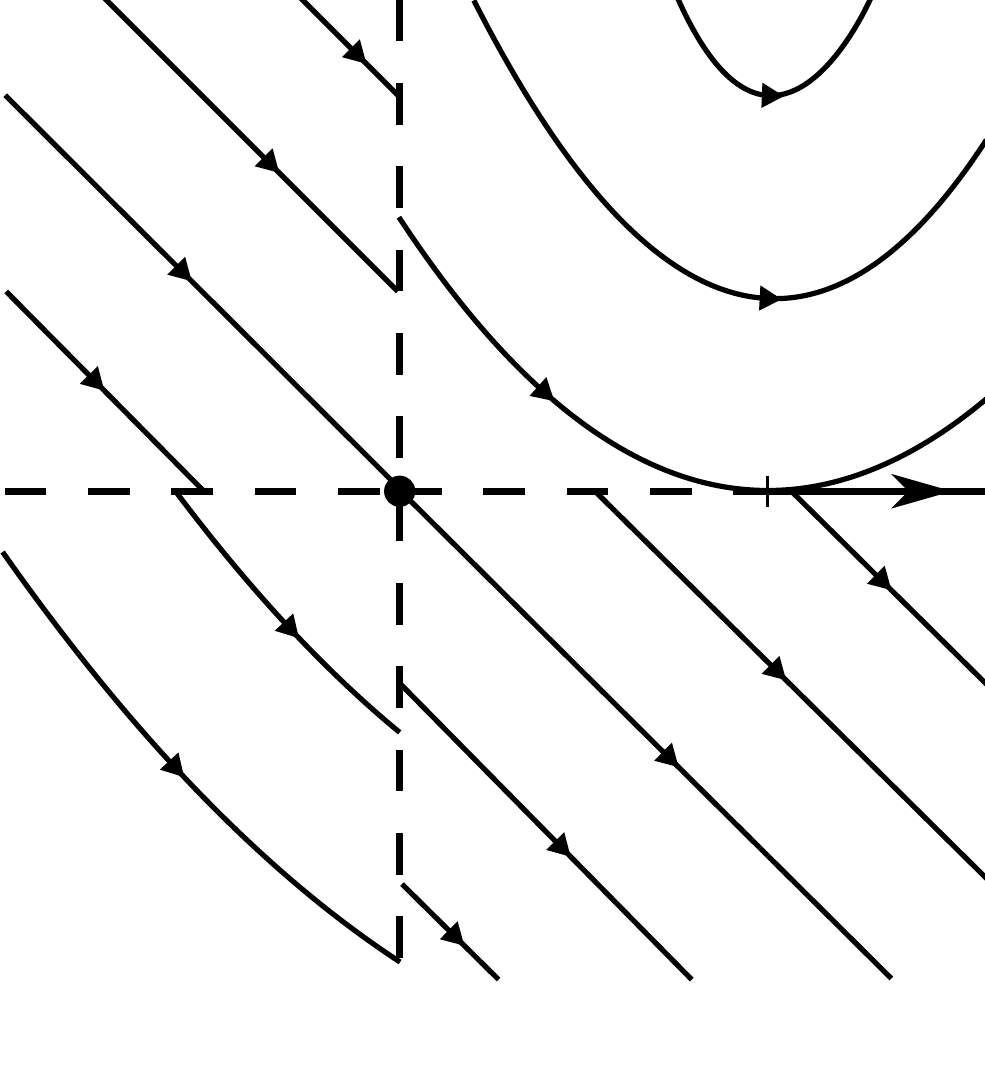}
			\vspace{0.2cm}
			\subfigure[\label{fig:F3} $Y_1 \cdot Y_2 (\0)<0$ and $Y_1(\0)<0$]{
				\def\svgscale{0.25}
				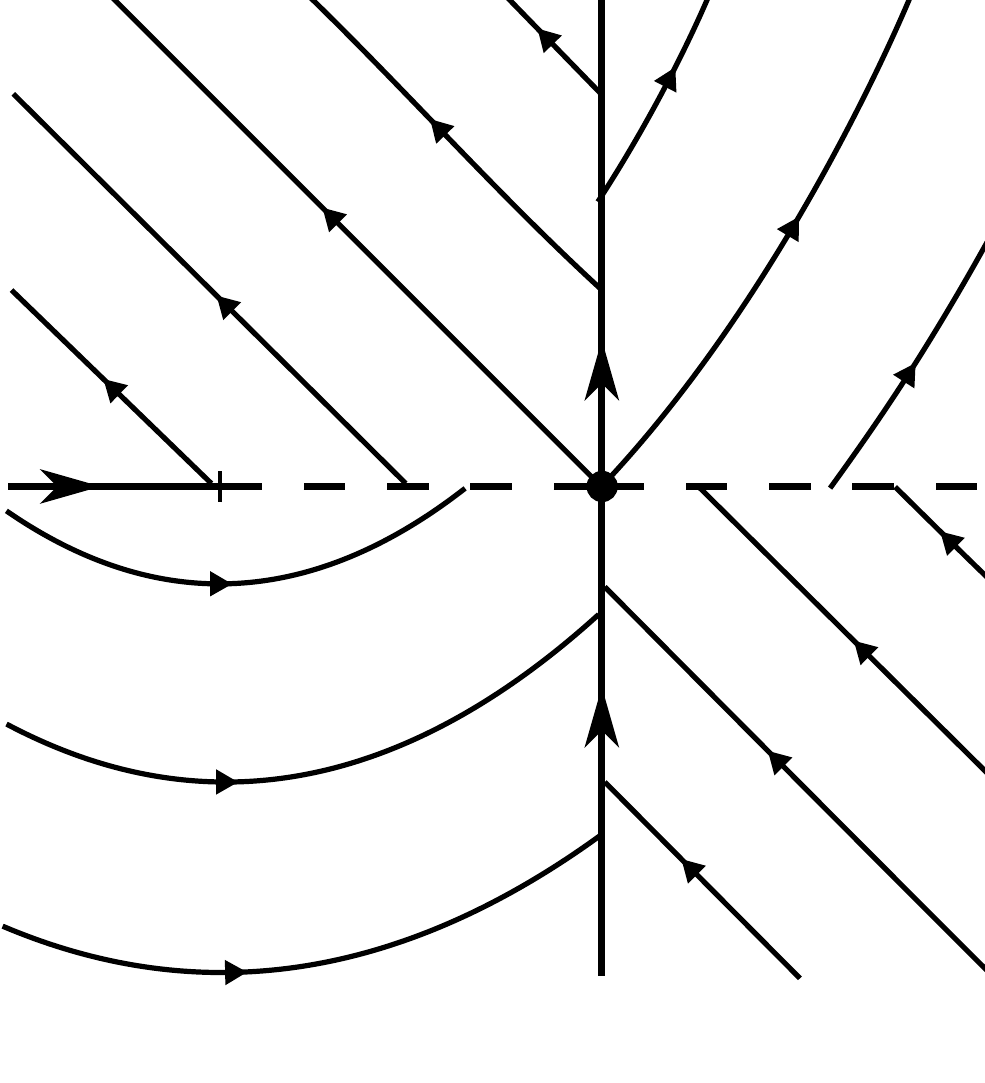 \hspace{0.5cm}
				\def\svgscale{0.25}
				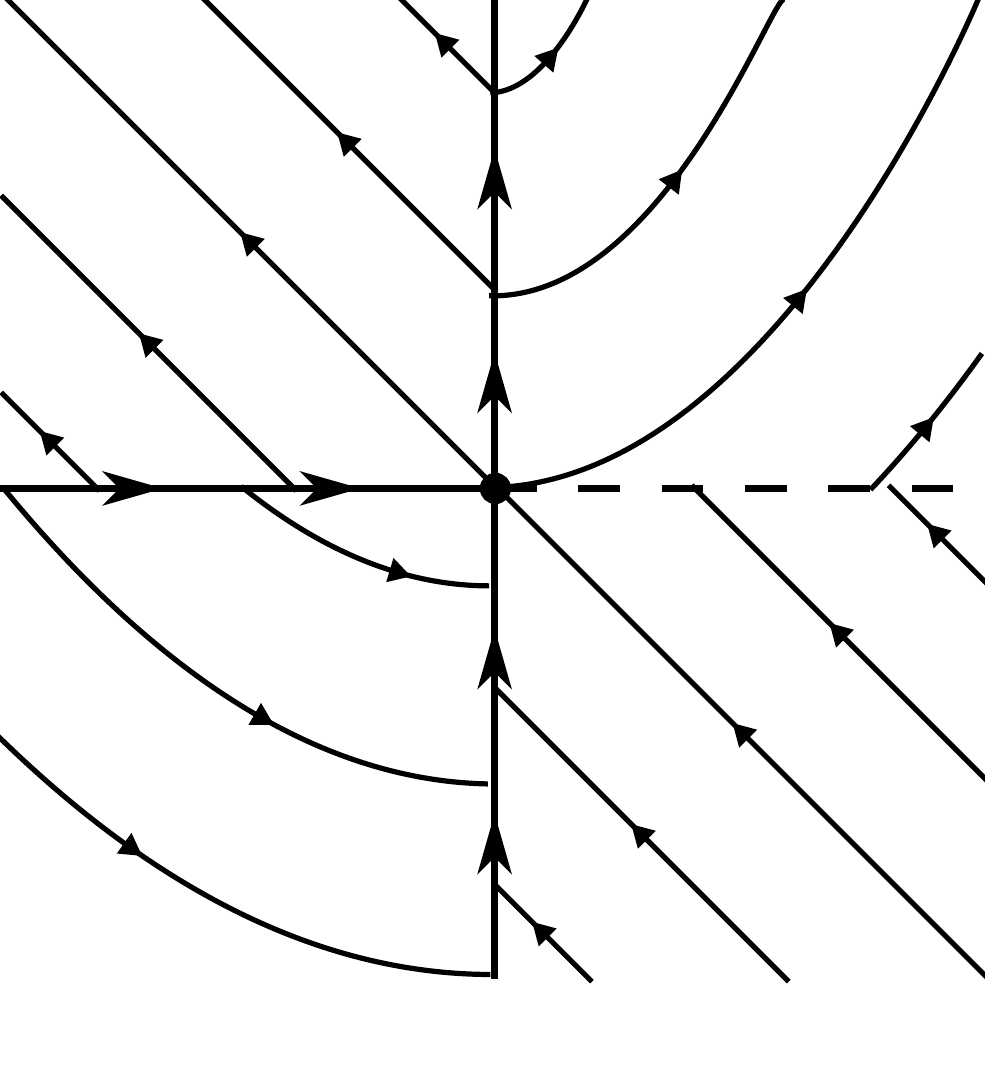 \hspace{0.5cm}
				\def\svgscale{0.25}
				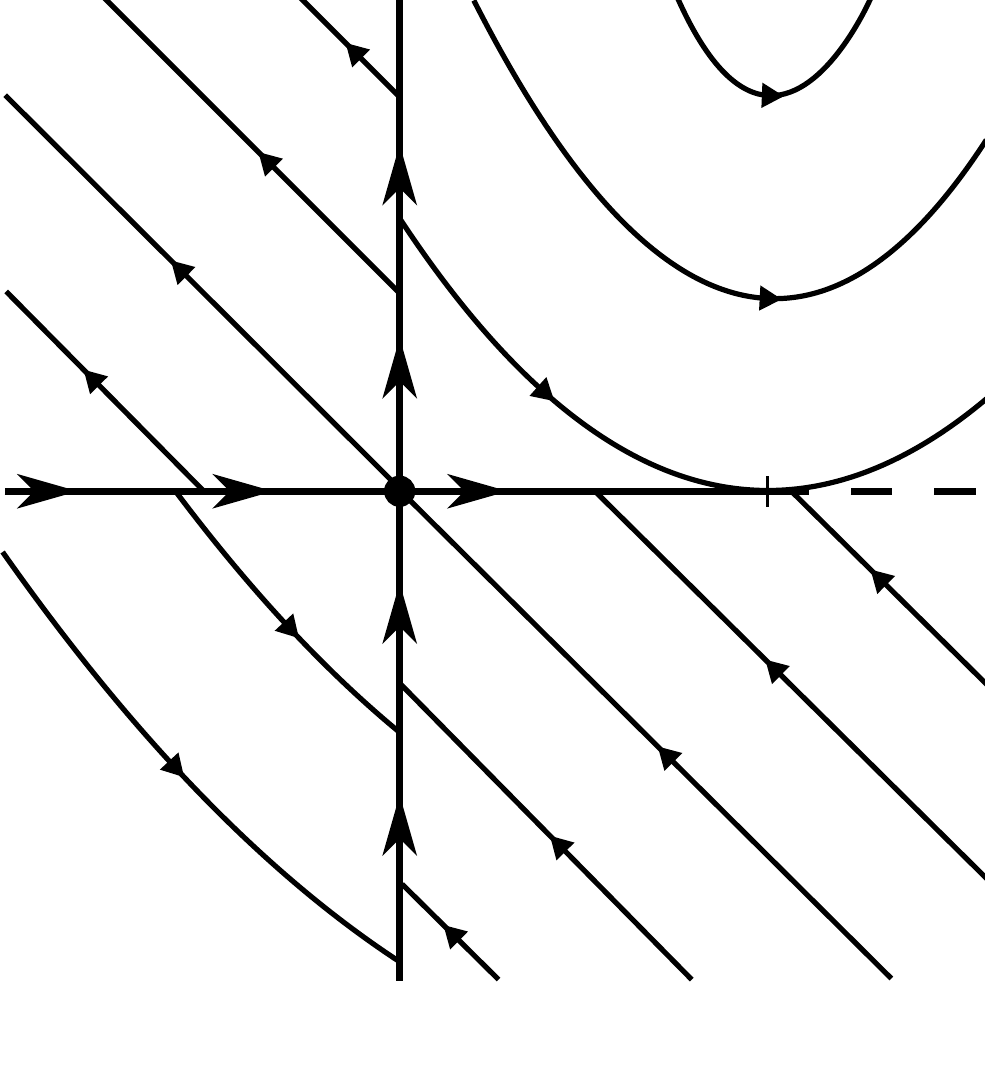}
			\vspace{0.2cm}

			\caption{The unfoldings for a regular-fold singularity satisfying $Y_1 \cdot Y_2 (\0)<0$.}
			\label{fig:F}
		\end{tiny}
	\end{figure}
		
	\end{subsection}
	
\end{section}

\section{Acknowledgments}

Tere M. Seara have been partially supported by the Spanish MINECO-FEDER Grant MTM2015-65715-P and the Catalan Grant 2014SGR504. Tere M-Seara  is also supported by the Russian Scientic Foundation grant 14-41-00044 and the European Marie Curie Action
FP7-PEOPLE-2012-IRSES: BREUDS. 

J. Larrosa has been supported by FAPESP grants 2011/22529-8 and 2014/13970-0 and the European Marie Curie Action FP7-PEOPLE-2012-IRSES: BREUDS. 

\printbibliography[title=Bibliography]

\end{document}